\documentclass[11pt,a4paper]{article}
\usepackage{amsmath,amssymb,amsthm,amsfonts}
\ifx\pdfoutput\undefined
  \usepackage[dvipdfmx]{graphicx}
\else
  \usepackage{graphicx}
\fi
\graphicspath{{./}}
\DeclareGraphicsExtensions{.pdf,.png,.jpg,.jpeg}
\usepackage{enumerate}
\usepackage{url}
\usepackage[margin=1in]{geometry}

\newtheorem{theorem}{Theorem}[section]
\newtheorem{proposition}[theorem]{Proposition}
\newtheorem{lemma}[theorem]{Lemma}
\newtheorem{corollary}[theorem]{Corollary}
\theoremstyle{definition}
\newtheorem{definition}[theorem]{Definition}
\newtheorem{assumption}[theorem]{Assumption}
\newtheorem{example}[theorem]{Example}
\newtheorem{remark}[theorem]{Remark}

\newcommand{\R}{\mathbb{R}}
\newcommand{\N}{\mathbb{N}}

\newcommand{\E}{\mathrm{E}}
\newcommand{\PP}{\mathrm{P}}
\newcommand{\Pt}{\widetilde{\PP}}
\newcommand{\Et}{\widetilde{\E}}
\newcommand{\F}{\mathcal{F}}
\newcommand{\T}{\mathcal{T}}
\newcommand{\A}{\mathcal{A}}
\newcommand{\B}{\mathcal{B}}
\newcommand{\Mcl}{\mathcal{M}}
\newcommand{\esup}{\operatorname*{ess\,sup}}
\newcommand{\med}{\operatorname{med}}
\newcommand{\dd}{\mathrm{d}}
\title{Multiple Stopping Options on a Geometric Random Walk}
\author{Katsunori Ano\thanks{Department of Human-centered Data Science, Bunkyo Gakuin University.
\texttt{k-ano@bgu.ac.jp}}}
\date{September 25, 2026}
\begin{document}
\maketitle
\begin{abstract}
This article develops a finite-horizon multiple-stopping framework and applies it to
	three American-style contracts on a geometric random walk in a
	Cox--Ross--Rubinstein market: an American put, a Russian option, and a
	floating-strike geometric-average Asian put.
The general problem is represented by recursively defined Snell envelopes, with
	unused exercise rights encoded by a cemetery time; this yields an ordered optimal
	exercise vector without requiring all rights to be exercised.
For the American put, a median representation of successive marginal values yields
	diminishing marginal values, nested exercise regions, and monotone exercise
	thresholds without relying on convexity of the marginal value.
After suitable state reductions, analogous marginal-value arguments give
	threshold-type optimal exercise rules for the Russian and geometric-average Asian
	options.  
Independent random maturity is also incorporated, and its effect on the
	corresponding stopping regions is identified.
\end{abstract}


%%%%%%%%%%%%%%%%%%%%%%%%%%%%%%%%%%%%%%%%%%%%%%%
\tableofcontents
%%%%%%%%%%%%%%%%%%%%%%%%%%%%%%%%%%%%%%%%%%%%%%%
\section{Introduction}
%%%%%%%%%%%%%%%%%%%%%%%%%%%%%%%%%%%%%%%%%%%%%%%
American-style contracts with multiple exercise rights naturally give rise to
	finite-horizon multiple stopping problems.
In the discrete-time setting considered here, the option holder has $m$ exercise
	rights, at most one right may be exercised at each date, and unused rights need
	not be exercised.
Our aim is to identify a common structural mechanism for deriving optimal
	multiple-exercise strategies across several option models driven by a geometric
	random walk.
The three applications are an American put, a Russian option, and a
	floating-strike geometric-average Asian put; in each case, we also consider an
	independent random maturity.

We maintain the convention that at most one right may be exercised at each date.
For an additive reward process $X_\nu$, the finite-horizon problem with $m$
	rights can be written schematically as
	\begin{equation}\label{eq:intro2}
 		V^{[m]}_0=\sup_{0\le \ell\le m}\ \sup_{0\le\sigma_1<\cdots<\sigma_\ell\le N}
   			\E\!\left[\sum_{j=1}^{\ell}X_{\sigma_j}\right].
	\end{equation}
Section~\ref{sec:general} recalls the recursive Snell-envelope construction,
	which reduces \eqref{eq:intro2} to a sequence of single-stopping problems.
Two points are important for a fully usable finite-horizon formulation.
First, unused rights must be represented without forcing exercise; we handle
	this by introducing a cemetery time with zero payoff.
Second, the optimal stopping times must be defined recursively so that the
	resulting exercise vector is ordered.
Theorem~\ref{thm:multiple} establishes the exact value representation and the
	optimality of this recursively constructed vector.
It serves as the basic verification result throughout the paper.

The paper is related to several strands of the multiple stopping literature.
Early discrete-time work on sequential multiple stopping includes
	Haggstrom~\cite{haggstrom} and Nikolaev~\cite{nikolaev}.
Carmona and Touzi~\cite{carmona} developed a Snell-envelope formulation
	for swing options, proved existence of multiple-exercise policies in a general setting,
	and gave a constructive solution in the perpetual Black--Scholes case together with
	a finite-horizon approximation procedure.
Carmona and Dayanik~\cite{carmonadayanik} studied multiple stopping for regular linear diffusions.
Kobylanski, Quenez, and Rouy-Mironescu~\cite{kobylanski} developed a general theory of
	multiple stopping.
A recent systematic treatment of discrete-time multiple stopping, including unilateral
	and multilateral formulations, is given by Sofronov and Szajowski~\cite{sofronovszajowski}.
For finite horizon swing puts with a positive refraction time,
	De Angelis and Kitapbayev~\cite{deangeliskitapbayev} characterized
	the continuous-time exercise regions in terms of free boundaries,
	in a setting in which all exercise rights must be exercised by maturity.
Ano~\cite{ano} derived an optimal sequence of stopping times for
	American put, Russian, and Asian options with multiple exercise rights 
	with a positive refraction time in the Black--Scholes model.
Numerical approaches include Meinshausen and Hambly~\cite{meinshausen} 
	and Bender and Schoenmakers~\cite{bender}.  
Preliminary discrete-time analyses of the double- and multiple-exercise American put appear in
	Oishi, Usui and Ano~\cite{oishi-usui-ano} and Oishi and Ano~\cite{oishi-ano}.
In the latter work, the extension to general numbers of exercise rights was tied to
a convexity property of the marginal value that was left unproved beyond the
double-exercise case.  The present paper replaces that convexity requirement by a
median representation together with a Lipschitz/single-crossing argument.  This
gives a unified marginal-value formulation for arbitrary numbers of rights and
extends the same structural viewpoint to the Russian and geometric-Asian settings,
as well as to independent random maturity.  
The Russian option goes back to Shepp and Shiryaev~\cite{shepp1,shepp2}; 
	for geometric-average Asian options see Kemna and Vorst~\cite{kemnavorst}, 
	and for American-style Asian stopping regions see, 
	for example, Dai and Kwok~\cite{daikwok}. 
General references on optimal stopping include Chow, Robbins and Siegmund~\cite{chow}, 
	Neveu~\cite{neveu}, Peskir and Shiryaev~\cite{peskir}, and Ano~\cite{anobook}.

The remainder of the paper is organised as follows.  
Section~\ref{sec:general} gives the recursive Snell envelope formulation. 
Section~\ref{sec:put} develops the marginal-value method for the American put, 
	and Section~\ref{sec:random} treats its random-maturity version.  
Section~\ref{sec:russian} studies the Russian option after the numeraire reduction, 
	including random maturity.  
Section~\ref{sec:asian} treats the geometric-average Asian put and its random-maturity extension.
%Section~\ref{sec:conclusion} concludes.  
Appendix~\ref{app:operator} collects the median-operation facts used repeatedly, and 
	Appendix~\ref{app:smoothfit} records the computations underlying the fixed-mesh smooth-fit discussion.
	
%%%%%%%%%%%%%%%%%%%%%%%%%%%%%%%%%%%%%%%%%%%%%%%
\section{Finite-horizon multiple stopping} %via recursive Snell envelopes}
\label{sec:general}
%%%%%%%%%%%%%%%%%%%%%%%%%%%%%%%%%%%%%%%%%%%%%%%
Throughout this section, time is \emph{calendar} time and is denoted by $\nu$.
Let $(\Omega,\F,\PP)$ be a probability space carrying a filtration
	$\mathbb{F}=(\F_\nu)_{\nu=0}^{N}$ with $N<\infty$, and let
	$X=(X_\nu)_{\nu=0}^{N}$ be an $\mathbb{F}$-adapted, integrable reward sequence.
We use the auxiliary non-exercise time (cemetery time) $\partial:=+\infty$.  
It represents the event that a right is never exercised, and we set
	\begin{equation}\label{eq:nplus}
  		\F_{\partial}:=\F_N,\qquad X_{\partial}:=0 .
	\end{equation}
Set $\mathbb{T}:=\{0,1,\dots,N\}\cup\{\partial\}$, ordered so that
	$\nu<\partial$ for every $\nu\le N$.  
Whenever a successor $\nu+1$ is written with $\nu\in\mathbb{T}$, we use the convention
	$N+1=\partial$ and $\partial\,+1=\partial$.  
For $\nu\in\mathbb{T}$ let $\T^{\nu}$ be the set of
	$\mathbb{F}$-stopping times with values in $\{\mu\in\mathbb{T}:\mu\ge\nu\}$.  
We write $\T:=\T^{0}$ and
	\begin{align*}
   		\bar X \;:=\; \max_{0\le\nu\le N}|X_\nu| ,
	\end{align*}
	which is integrable because $N<\infty$; we record this as
	\begin{equation}\label{eq:A1}
  		\textbf{(A1)}\qquad \E[\bar X]<\infty .
	\end{equation}

%%%%%%%%%%%%%%%%%%%%%%%%%%%%%%%%%%%%%%%%%%%%%%%
\subsection{Single stopping}
%%%%%%%%%%%%%%%%%%%%%%%%%%%%%%%%%%%%%%%%%%%%%%%
We first recall the classical result of Snell~\cite{snell} in the form in which we shall use it.
%%%%%%%%%%%%%%%%%%%%%%%%%%%%%%%%%%%%%%%%%%%%%%%
\begin{theorem}%[Snell]
\label{thm:snell}
Assume \eqref{eq:A1}.  
Define recursively
	\begin{equation*}
  		V_{\partial}:=0,\qquad
  		V_\nu:=\max\big\{X_\nu,\ \E[V_{\nu+1}\mid\F_\nu]\big\},\quad \nu=N,N-1,\dots,0 ,
	\end{equation*}
	with the convention $V_{N+1}:=V_{\partial}=0$.  
Then:
	\begin{enumerate}[(a)]
	\item $(V_\nu)$ is the smallest supermartingale dominating $(X_\nu)$;
	\item for every $\sigma\in\T$,
      		$V_\sigma=\esup_{\tau\in\T^{\sigma}}\E[X_\tau\mid\F_\sigma]$ a.s.;
	\item $\tau^{*}(\sigma):=\min\{\nu\ge\sigma:\ V_\nu=X_\nu\}$
      		(with $\min\emptyset:=\partial$) belongs to $\T^{\sigma}$ and
      		$V_\sigma=\E[X_{\tau^{*}(\sigma)}\mid\F_\sigma]$ a.s.;
	\item the stopped process $(V_{\tau^{*}(\sigma)\wedge\nu})_{\nu\ge\sigma}$ is a
      		martingale.
	\end{enumerate}
\end{theorem}
%%%%%%%%%%%%%%%%%%%%%%%%%%%%%%%%%%%%%%%%%%%%%%%
Part (b) in the form ``conditionally on $\F_\sigma$ for an arbitrary stopping
	time $\sigma$'' is what makes the induction of Theorem~\ref{thm:multiple} work,
	and this is why we have stated it that way; see Neveu \cite[Ch.\ VI]{neveu} or
	Peskir and Shiryaev \cite[Ch.\ I]{peskir}.

%% ---------------------------------------------------------------------
\subsection{Multiple stopping}
%% ---------------------------------------------------------------------
%%%%%%%%%%%%%%%%%%%%%%%%%%%%%%%%%%%%%%%%%%%%%%%
\begin{definition}\label{def:admissible}
For $m\in\N$ and $\sigma\in\T$ let $\T^{[m]}_\sigma$ denote the set of vectors
	$\vec\tau=(\tau_m,\tau_{m-1},\dots,\tau_1)$ of elements of $\T^{\sigma}$ for
	which there exists $\ell\in\{0,1,\dots,m\}$ such that the finite coordinates
	are exactly
	\begin{align*}
 	 	\tau_m,\tau_{m-1},\dots,\tau_{m-\ell+1}.
	\end{align*}
If $\ell\ge1$ they satisfy
	\begin{align*}
  		\sigma\le\tau_m<\tau_{m-1}<\cdots<\tau_{m-\ell+1}\le N,
	\end{align*}
	while the remaining coordinates
	$\tau_{m-\ell},\dots,\tau_1$ (if any) are all equal to $\partial$.  
Thus the finite entries form the chronologically ordered block of exercised rights, 
	while all unused rights are placed at the non-exercise time.  
We set $X_{\vec\tau}:=\sum_{i=1}^{m}X_{\tau_i}$; 
	by \eqref{eq:nplus} unexercised rights contribute nothing.
\end{definition}
%%%%%%%%%%%%%%%%%%%%%%%%%%%%%%%%%%%%%%%%%%%%%%%
The nested construction is
	\begin{equation*}
  		V^{[0]}\equiv 0,\qquad
  		X^{[i]}_\nu := X_\nu+\E\big[V^{[i-1]}_{\nu+1}\mid\F_\nu\big],\qquad
  		V^{[i]} := \text{Snell envelope of } X^{[i]},
	\end{equation*}
	for $i=1,\dots,m$, with $X^{[i]}_{\partial}:=0$ and $V^{[i]}_{\partial}:=0$.  
In particular $X^{[1]}=X$ and $V^{[1]}$ is the ordinary Snell envelope.
%%%%%%%%%%%%%%%%%%%%%%%%%%%%%%%%%%%%%%%%%%%%%%%
\begin{lemma}%[Integrability]
\label{lem:integrability}
Assume \eqref{eq:A1} and $N<\infty$.  
Then for every $i\in\{1,\dots,m\}$,
	\begin{align*}
   		|V^{[i]}_\nu|\le \E\big[\,i\,\bar X\mid \F_\nu\big]\quad\text{a.s.},
  		 \qquad
   		\E\Big[\max_{0\le\nu\le N}\big|X^{[i]}_\nu\big|\Big]\;\le\;(N+1)\,i\,\E[\bar X]\;<\;\infty .
	\end{align*}
Consequently Theorem~\ref{thm:snell} applies to each $X^{[i]}$.
\end{lemma}
%%%%%%%%%%%%%%%%%%%%%%%%%%%%%%%%%%%%%%%%%%%%%%%
\begin{proof}
For the first bound, induct on $i$.  
For $i=1$, 
	$$|V^{[1]}_\nu|=|\esup_{\tau\ge\nu}\E[X_\tau\mid\F_\nu]|\le \E[\bar X\mid\F_\nu]$$ 
	by the monotonicity of conditional expectation.  
If the bound holds for $i-1$, then
	$$|X^{[i]}_\nu|\le |X_\nu|+\E[(i-1)\bar X\mid\F_\nu]\le \E[i\bar X\mid\F_\nu],$$
	whence 
	$$|V^{[i]}_\nu|\le\esup_{\tau\ge\nu}\E[|X^{[i]}_\tau|\mid\F_\nu]
	\le \E[i\bar X\mid\F_\nu].$$ 
 For the second bound, since $N<\infty$,
	\begin{align*}
  		\E\Big[\max_{\nu\le N}|X^{[i]}_\nu|\Big]
  		\;\le\;\sum_{\nu=0}^{N}\E\big[|X^{[i]}_\nu|\big]
  		\;\le\;\sum_{\nu=0}^{N} i\,\E[\bar X]
  		\;=\;(N+1)\,i\,\E[\bar X] . \qedhere
	\end{align*}
\end{proof}

%%%%%%%%%%%%%%%%%%%%%%%%%%%%%%%%%%%%%%%%%%%%%%%
\begin{definition}%[Recursive optimal stopping times]
\label{def:recursive}
Fix $\sigma\in\T$ and $m\in\N$.  Define
	\begin{equation}\label{eq:taustar}
  		\tau^{*}_m:=\min\big\{\nu\ge\sigma:\ V^{[m]}_\nu=X^{[m]}_\nu\big\},
 	 	\qquad
  		\tau^{*}_{i}:=\min\big\{\nu\ge\tau^{*}_{i+1}+1:\ V^{[i]}_\nu=X^{[i]}_\nu\big\}
	\end{equation}
	for $i=m-1,m-2,\dots,1$, with $\min\emptyset:=\partial$ and the successor
	convention stated above.
\end{definition}
%%%%%%%%%%%%%%%%%%%%%%%%%%%%%%%%%%%%%%%%%%%%%%%
The recursion in \eqref{eq:taustar} is essential.  
Had we defined $\tau^{*}_i$ as the first time {after $\sigma$} 
	at which $V^{[i]}$ meets $X^{[i]}$, the
	vector $(\tau^{*}_m,\dots,\tau^{*}_1)$ would in general fail to be increasing,
	and would therefore fail to be admissible in the sense of Definition~\ref{def:admissible}.

%%%%%%%%%%%%%%%%%%%%%%%%%%%%%%%%%%%%%%%%%%%%%%%
\begin{theorem}%[Recursive Snell-envelope representation]
\label{thm:multiple}
Assume \eqref{eq:A1} and $N<\infty$.  
Then for every $m\in\N$ and every
$\sigma\in\T$,
	\begin{equation}\label{eq:multiplevalue}
  		V^{[m]}_\sigma\;=\;\esup_{\vec\tau\in\T^{[m]}_\sigma}\E\big[X_{\vec\tau}\mid\F_\sigma\big]
  	\qquad\text{a.s.},
	\end{equation}
	and the vector $\vec\tau^{\,*}=(\tau^{*}_m,\dots,\tau^{*}_1)$ of
	Definition~\ref{def:recursive} is admissible and optimal:
	$V^{[m]}_\sigma=\E[X_{\vec\tau^{\,*}}\mid\F_\sigma]$ a.s.
\end{theorem}
%%%%%%%%%%%%%%%%%%%%%%%%%%%%%%%%%%%%%%%%%%%%%%%
\begin{proof}
We argue by induction on $m$; the case $m=1$ is Theorem~\ref{thm:snell}(b),(c).
Let $m\ge2$ and assume the statement for $m-1$ (for every stopping time 
	in place of $\sigma$).

\emph{Step 1: ``$\ge$'' in \eqref{eq:multiplevalue}.}
Let $\vec\tau=(\tau_m,\dots,\tau_1)\in\T^{[m]}_\sigma$.  
On $\{\tau_m\le N\}$ the truncated vector $(\tau_{m-1},\dots,\tau_1)$ belongs to
	$\T^{[m-1]}_{\tau_m+1}$, so the induction hypothesis applied with $\tau_m+1$ in
	place of $\sigma$ gives
	\begin{align*}
   		\E\Big[\sum_{i=1}^{m-1}X_{\tau_i}\;\Big|\;\F_{\tau_m+1}\Big]
   		\;\le\; V^{[m-1]}_{\tau_m+1}\qquad\text{a.s.}
	\end{align*}
On $\{\tau_m=\partial\}$ all $\tau_i=\partial$ and both sides vanish.  
Taking conditional expectations,
	\begin{align*}
  		 \E\big[X_{\vec\tau}\mid\F_\sigma\big]
  		 \;\le\;\E\Big[X_{\tau_m}+\E\big[V^{[m-1]}_{\tau_m+1}\mid\F_{\tau_m}\big]
       		 \;\Big|\;\F_\sigma\Big]
   		\;=\;\E\big[X^{[m]}_{\tau_m}\mid\F_\sigma\big]
   		\;\le\; V^{[m]}_\sigma ,
	\end{align*}
	the last step by Theorem~\ref{thm:snell}(b) applied to $X^{[m]}$, which is
	legitimate by Lemma~\ref{lem:integrability}.  
Taking the essential supremum over $\vec\tau$ gives ``$\ge$''.

\emph{Step 2: ``$\le$'' in \eqref{eq:multiplevalue}, and optimality.}
By Theorem~\ref{thm:snell}(c) applied to $X^{[m]}$ and by the definition of
	$\tau^{*}_m$,
	\begin{align*}
   		V^{[m]}_\sigma=\E\big[X^{[m]}_{\tau^{*}_m}\mid\F_\sigma\big]
   		=\E\Big[X_{\tau^{*}_m}+\E\big[V^{[m-1]}_{\tau^{*}_m+1}\mid\F_{\tau^{*}_m}\big]
     		\;\Big|\;\F_\sigma\Big]
   		=\E\Big[X_{\tau^{*}_m}+V^{[m-1]}_{\tau^{*}_m+1}\;\Big|\;\F_\sigma\Big].
	\end{align*}
Now apply the induction hypothesis with $\sigma$ replaced by
	$\tau^{*}_m+1$.  
By Definition~\ref{def:recursive} the optimal vector for the
	$(m-1)$-fold problem started at $\tau^{*}_m+1$ is exactly
	$(\tau^{*}_{m-1},\dots,\tau^{*}_1)$, whence
	\begin{align*}
  		V^{[m-1]}_{\tau^{*}_m+1}
  		=\E\Big[\sum_{i=1}^{m-1}X_{\tau^{*}_i}\;\Big|\;\F_{\tau^{*}_m+1}\Big] .
	\end{align*}
Substituting, $V^{[m]}_\sigma=\E[X_{\vec\tau^{\,*}}\mid\F_\sigma]$.
By construction, consecutive finite entries of $(\tau_m^*,\dots,\tau_1^*)$ are strictly increasing.  
Once one entry equals
	$\partial$, the successor convention forces every later entry to equal $\partial$.
Hence $\vec\tau^{\,*}\in\T^{[m]}_\sigma$, so the right-hand side of
	\eqref{eq:multiplevalue} is at least $V^{[m]}_\sigma$.  
Together with Step~1 this proves \eqref{eq:multiplevalue} and the optimality of $\vec\tau^{\,*}$.
\end{proof}

%%%%%%%%%%%%%%%%%%%%%%%%%%%%%%%%%%%%%%%%%%%%%%%
\begin{corollary}\label{cor:monotone_m}
$V^{[0]}_\nu\le V^{[1]}_\nu\le V^{[2]}_\nu\le\dots$, and if $X\ge0$ then
$V^{[m]}_\nu\le m\,\E[\bar X\mid\F_\nu]$.
\end{corollary}
%%%%%%%%%%%%%%%%%%%%%%%%%%%%%%%%%%%%%%%%%%%%%%%
\begin{proof}
Immediate from \eqref{eq:multiplevalue}: if
	$(\tau_{m-1},\dots,\tau_1)$ is admissible for $m-1$ rights, define an
	$m$-right vector by shifting the labels of its finite block one level upward and
	placing one additional unused right at the terminal end of the vector.  
In particular, if all $m-1$ rights are used, take
	$(\widetilde\tau_m,\dots,\widetilde\tau_2)
		=(\tau_{m-1},\dots,\tau_1)$ and $\widetilde\tau_1=\partial$.
The total reward is unchanged because $X_{\partial}=0$.  
The upper bound follows from $0\le X_{\tau_i}\le\bar X$.
\end{proof}

%%%%%%%%%%%%%%%%%%%%%%%%%%%%%%%%%%%%%%%%%%%%%%%
\subsection{The Markovian case}\label{sec:markov}
%%%%%%%%%%%%%%%%%%%%%%%%%%%%%%%%%%%%%%%%%%%%%%%
Let $(Z_\nu)_{\nu=0}^{N}$ be a time-homogeneous Markov chain on a state space
	$E$ with transition kernel $P$, let $\F_\nu=\sigma(Z_0,\dots,Z_\nu)$, and let
	the reward be $X_\nu=\alpha^{\nu}g(Z_\nu)$ for a measurable function $g$ with $g(z)\ge0$ and a
	discount factor $\alpha\in(0,1]$.  
Then the Snell envelopes admit the Markov representation
	\begin{align*}
   		V^{[m]}_\nu=\alpha^\nu V^{[m]}_{N-\nu}(Z_\nu), \qquad n:=N-\nu ,
	\end{align*}
	for deterministic functions $z\mapsto V^{[m]}_n(z)$.  
Thus $V^{[m]}_n(z)$ is the value
	measured in time-$\nu$ units when $n$ periods remain, $m$ rights are in hand and
	the current state is $z$.  
The dynamic programming equations read
	\begin{equation}\label{eq:dpgeneral}
  		V^{[m]}_0(z)=g(z)\ (m\ge1),\qquad V^{[0]}_n(z)\equiv0,
	\end{equation}
	\begin{equation}\label{eq:dpgeneral2}
  		V^{[m]}_n(z)=\max\Big\{\,g(z)+\alpha\,\E_z\big[V^{[m-1]}_{n-1}(Z_1)\big],
             		 \;\alpha\,\E_z\big[V^{[m]}_{n-1}(Z_1)\big]\Big\},
  			\qquad n\ge1 .
	\end{equation}
By Theorem~\ref{thm:multiple}, the optimal exercise rule in \emph{calendar}
	time is
	\begin{align}
  		\sigma^{*}_m&=\min\big\{\nu\in\{0,\dots,N\}: Z_\nu\in D^{[m]}_{N-\nu}\big\},
  			\label{eq:calendar}\\
 		 \sigma^{*}_{i}&=\min\big\{\nu\in\{0,\dots,N\}: \nu>\sigma^{*}_{i+1},\
       		Z_\nu\in D^{[i]}_{N-\nu}\big\}.\notag
			\end{align}
$i=m-1,\dots,1$, where, for $n\ge1$,
	\begin{align*}
 		D^{[i]}_n:=\Big\{z:\  g(z)+\alpha\E_z[V^{[i-1]}_{n-1}(Z_1)]
  			\ge \alpha\E_z[V^{[i]}_{n-1}(Z_1)]\Big\},
	\end{align*}
	and $D^{[i]}_0:=E$.  
At a tie either action is optimal; later, for the put, we
	will choose continuation at zero-payoff out-of-the-money ties.  
We use the convention $\min\emptyset:=\partial$.  
The finite stopping times produced by
	\eqref{eq:calendar} are strictly increasing; any unused rights are placed at $\partial$.
We use $n$ for the remaining time and $\nu$ for calendar time throughout.

%%%%%%%%%%%%%%%%%%%%%%%%%%%%%%%%%%%%%%%%%%%%%%%
\section{American put}%: marginal values and optimal exercise boundaries}
\label{sec:put}
%%%%%%%%%%%%%%%%%%%%%%%%%%%%%%%%%%%%%%%%%%%%%%%
\subsection{The model and the one-step operator}
%%%%%%%%%%%%%%%%%%%%%%%%%%%%%%%%%%%%%%%%%%%%%%%
Let $r\ge0$ be the one-period interest rate, $\alpha:=(1+r)^{-1}\in(0,1]$, and
	let $\lambda>1$ satisfy
	\begin{equation}\label{eq:noarb}
  		\lambda^{-1}<1+r<\lambda ,
	\end{equation}
	which is the no-arbitrage condition.  
The stock price is the geometric random walk
	\begin{align*}
  		S_\nu=S_0\,\lambda^{\varepsilon_1+\dots+\varepsilon_\nu},
  		\qquad \varepsilon_\nu\in\{-1,+1\}\ \text{i.i.d.},
	\end{align*}
	and, from this point through Section~\ref{sec:random}, $\PP$ denotes the unique
	martingale measure, under which
	\begin{equation}\label{eq:pq}
  		p:=\PP(\varepsilon_\nu=+1)=\frac{\alpha^{-1}-\lambda^{-1}}{\lambda-\lambda^{-1}},
  			\qquad
  		q:=\PP(\varepsilon_\nu=-1)=\frac{\lambda-\alpha^{-1}}{\lambda-\lambda^{-1}}=1-p,
	\end{equation}
	both in $(0,1)$ by \eqref{eq:noarb}.  
The payoff of the put with strike $K>0$ is
	\begin{equation*}
  		g(x):=(K-x)^{+},\qquad x>0 .
	\end{equation*}

Define the one-step (discounted) valuation operator, acting on functions
	$\varphi:(0,\infty)\to[0,\infty)$,
	\begin{equation}\label{eq:operator}
  		(\A\varphi)(x):=\alpha\big[\,p\,\varphi(\lambda x)+q\,\varphi(\lambda^{-1}x)\,\big].
	\end{equation}
The following identity is used constantly and is simply the martingale property
	of the discounted price:
	\begin{equation}\label{eq:gain}
  		\alpha\big(p\lambda+q\lambda^{-1}\big)=\alpha(1+r)=1 .
	\end{equation}
%%%%%%%%%%%%%%%%%%%%%%%%%%%%%%%%%%%%%%%%%%%%%%%
\begin{definition}\label{def:M}
Let
	\begin{align*}
 	 	\Mcl:=\Big\{\varphi:(0,\infty)\to[0,\infty)\ \Big|\
  		\varphi\ \text{is nonincreasing and}\ |\varphi(x)-\varphi(y)|\le|x-y|
  		\ \ \forall x,y>0\Big\} ,
	\end{align*}
	the class of nonnegative, nonincreasing, $1$-Lipschitz functions.  
Equivalently,$\varphi\in\Mcl$ if and only if $\varphi\ge0$, $x\mapsto\varphi(x)$ is
	nonincreasing and $x\mapsto x+\varphi(x)$ is nondecreasing; for a differentiable
	$\varphi$ this reads $-1\le\varphi'\le0$.  
All functions occurring below are continuous and piecewise affine, 
	so we shall freely use the derivative notation
	for the (existing) one-sided derivatives.
\end{definition}
%%%%%%%%%%%%%%%%%%%%%%%%%%%%%%%%%%%%%%%%%%%%%%%
\begin{lemma}\label{lem:operator}
Let $\varphi,\psi:(0,\infty)\to[0,\infty)$.
\begin{enumerate}[(i)]
	\item If $\varphi\le\psi$ then $\A\varphi\le\A\psi$.
	\item If $\varphi$ is nonincreasing, so is $\A\varphi$.  If $\varphi$ is convex,
      		so is $\A\varphi$.
	\item If $\varphi$ is $L$-Lipschitz, then $\A\varphi$ is $L$-Lipschitz.  In
      		particular $\A(\Mcl)\subset\Mcl$.
	\item $\Mcl$ is a convex set, closed under $\max$, $\min$, $\med$ and pointwise
      		limits, and $g(\cdot)\in\Mcl$.
	\item If $\varphi(x)=a-bx$ on an interval containing $\lambda x$ and
      		$\lambda^{-1}x$, then $(\A\varphi)(x)=\alpha a-bx$.
\end{enumerate}
\end{lemma}
%%%%%%%%%%%%%%%%%%%%%%%%%%%%%%%%%%%%%%%%%%%%%%%
\begin{proof}
(i), (ii) are immediate.  (iii): for $x>y$,
$|(\A\varphi)(x)-(\A\varphi)(y)|\le \alpha\big[pL\lambda+qL\lambda^{-1}\big](x-y)
	=L(x-y)$ by \eqref{eq:gain}.  (iv): $\med\{a,b,c\}
	=\max\{\min\{a,b\},\min\{b,c\},\min\{c,a\}\}$, and both $\max$ and $\min$ of
	nonincreasing $1$-Lipschitz functions are nonincreasing and $1$-Lipschitz;
	$g(x)$ is nonnegative, nonincreasing and $1$-Lipschitz.  (v) is \eqref{eq:gain}
	again.
\end{proof}

Part (iii) of Lemma~\ref{lem:operator} is the reason why $1$ is the natural
	Lipschitz constant here: by \eqref{eq:gain} the operator $\A$ is neither a
	contraction nor an expansion.  
This is in contrast with the Russian option of
	Section~\ref{sec:russian}, where the analogous operator has gain $\alpha<1$.

%%%%%%%%%%%%%%%%%%%%%%%%%%%%%%%%%%%%%%%%%%%%%%%
\subsection{Dynamic programming and the median identity}
%%%%%%%%%%%%%%%%%%%%%%%%%%%%%%%%%%%%%%%%%%%%%%%
Let $n=N-\nu$ be the remaining time and let
	$V^{[m]}_n(x)$ denote the value of the option with $m$ rights, $n$ periods to
	maturity and current price $x$.  
By \eqref{eq:dpgeneral}--\eqref{eq:dpgeneral2},
	\begin{equation*}
  		V^{[m]}_0(x)=g(x)\ \ (m\ge1),\qquad V^{[0]}_n(x)\equiv0\ \ (n\ge0),
	\end{equation*}
	\begin{equation}\label{eq:dp2}
  		V^{[m]}_n(x)=\max\big\{\,g(x)+(\A V^{[m-1]}_{n-1})(x),\ (\A V^{[m]}_{n-1})(x)\,\big\},
  		\qquad n\ge1,\ m\ge1 .
	\end{equation}
Set
	\begin{equation}\label{eq:DeltaVf}
  		\Delta V^{[m]}_n(x):=V^{[m]}_n(x)-V^{[m-1]}_n(x),
 	 	\quad
  		f^{[m]}_n(x):=(\A \Delta V^{[m]}_{n-1})(x)\ \ (n\ge1),\quad f^{[m]}_0(x):=0 ,
	\end{equation}
	together with the convention
	\begin{equation}\label{eq:fzero}
  		f^{[0]}_n(x):\equiv+\infty .
	\end{equation}
Since $(\A V^{[m]}_{n-1})(x)=(\A V^{[m-1]}_{n-1})(x)+f^{[m]}_n(x)$, equation 
	\eqref{eq:dp2} can be rewritten in the form which we shall use exclusively:
\begin{equation}\label{eq:dp3}
  	V^{[m]}_n(x)
	=\max\underbrace{\big\{g(x),\ f^{[m]}_n(x)\big\}}_{\text{exercise or continue}}
		+(\A V^{[m-1]}_{n-1})(x) \qquad n\ge1,\ m\ge1 .
\end{equation}
Thus exercise is an optimal action whenever $g(x)\ge f^{[m]}_n(x)$, while
	continuation is optimal whenever $g(x)\le f^{[m]}_n(x)$.  
At equality both actions are optimal.  
For the put we shall break the economically irrelevant tie $g(x)=f_n^{[m]}(x)=0$ by choosing continuation.  
The function $f^{[m]}_n(x)$ is the continuation premium attached to the $m$-th right.  
Note that \eqref{eq:dp3} also holds for $n=0$ with the
	convention \eqref{eq:DeltaVf}, since $f^{[m]}_0(x)=0\le g(x)$ and
	$(\A V^{[m-1]}_{-1})(x):=0$.

%%%%%%%%%%%%%%%%%%%%%%%%%%%%%%%%%%%%%%%%%%%%%%%
\begin{lemma}%[Saturation]
\label{lem:saturation}
For every $n\ge0$ and $m\ge n+1$ we have $V^{[m]}_n(x)=V^{[n+1]}_n(x)$.  
Consequently
	$\Delta V^{[m]}_n(x)=0$ for $m\ge n+2$ and $f^{[m]}_n(x)=0$ for $m\ge n+1$.
\end{lemma}
%%%%%%%%%%%%%%%%%%%%%%%%%%%%%%%%%%%%%%%%%%%%%%%
\begin{proof}
Induction on $n$.  For $n=0$, $V^{[m]}_0(x)=g(x)$ for all $m\ge1$.  Let $n\ge1$ and
	$m\ge n+1$.  
Then $m-1\ge n$ and $m\ge n$, so by the induction hypothesis, 
	$V^{[m-1]}_{n-1}(x)=V^{[n]}_{n-1}(x)=V^{[m]}_{n-1}(x)$; hence, by \eqref{eq:dp2},
	$V^{[m]}_n(x)=\max\{g(x)+(\A V^{[n]}_{n-1})(x),(\A V^{[n]}_{n-1})(x)\}=g(x)+(\A V^{[n]}_{n-1})(x)$,
	which does not depend on $m$.  
The two consequences are immediate.
\end{proof}

Lemma~\ref{lem:saturation} formalises the obvious fact that with $n$ periods to
	go there are only $n+1$ exercise dates left, so that more than $n+1$ rights are
	worthless; it will replace the informal argument usually given for the identity
	$x^{[m]*}_n=K$, $n\le m-1$.
%%%%%%%%%%%%%%%%%%%%%%%%%%%%%%%%%%%%%%%%%%%%%%%
\begin{lemma}%[Median identity]
\label{lem:median}
For every $n\ge0$ and $m\ge2$, if $f^{[m]}_n(x)\le f^{[m-1]}_n(x)$ then
	\begin{align}\label{eq:median}
  		\Delta V^{[m]}_n(x) = \med\big\{\,f^{[m]}_n(x),\ g(x),\ f^{[m-1]}_n(x)\,\big\} 
 			=\min\Big\{f^{[m-1]}_n(x),\ \max\big\{g(x),\ f^{[m]}_n(x)\big\}\Big\} .
	\end{align}
For $m=1$ one has $\Delta V^{[1]}_n(x)=V^{[1]}_n(x)=\max\{g(x),f^{[1]}_n(x)\}$, 
	which is \eqref{eq:median} with the convention \eqref{eq:fzero}.
\end{lemma}
%%%%%%%%%%%%%%%%%%%%%%%%%%%%%%%%%%%%%%%%%%%%%%%
\begin{proof}
Apply \eqref{eq:dp3} at levels $m$ and $m-1$ and subtract:
	\begin{align*}
  		\Delta V^{[m]}_n(x)&=\max\{g(x),f^{[m]}_n(x)\}-\max\{g(x),f^{[m-1]}_n(x)\}
  		+\big(\A\big(V^{[m-1]}_{n-1}-V^{[m-2]}_{n-1}\big)\big)(x)\\
 		 &= \max\{g(x),f^{[m]}_n(x)\}-\max\{g(x),f^{[m-1]}_n(x)\}+f^{[m-1]}_n(x) .
	\end{align*}
	where the last equality is the definition \eqref{eq:DeltaVf} of $f^{[m-1]}_n(x)$.
Fix $x$ and distinguish three cases, using $f^{[m]}_n(x)\le f^{[m-1]}_n(x)$.
If $g(x)\ge f^{[m-1]}_n(x)$ the right-hand side equals
	$g(x)-g(x)+f^{[m-1]}_n(x)=f^{[m-1]}_n(x)$.  
If $f^{[m]}_n(x)\le g(x)\le f^{[m-1]}_n(x)$ it equals
	$g(x)-f^{[m-1]}_n(x)+f^{[m-1]}_n(x)=g(x)$.  
If $g(x)\le f^{[m]}_n(x)$ it equals
	$f^{[m]}_n(x)-f^{[m-1]}_n(x)+f^{[m-1]}_n(x)=f^{[m]}_n(x)$. 
In all three cases the value is the median of the three numbers.  
The case $m=1$ is \eqref{eq:dp3} with $(\A V^{[0]}_{n-1})(x)=0$.
\end{proof}

%%%%%%%%%%%%%%%%%%%%%%%%%%%%%%%%%%%%%%%%%%%%%%%
\subsection{Structural properties}
%%%%%%%%%%%%%%%%%%%%%%%%%%%%%%%%%%%%%%%%%%%%%%%
%%%%%%%%%%%%%%%%%%%%%%%%%%%%%%%%%%%%%%%%%%%%%%%
\begin{proposition}%[Marginal-value structure for the American put]
\label{prop:structure}
For all $n\ge0$ and $m\ge1$:
\begin{enumerate}[(i)]
	\item $\Delta V^{[m]}_n(\cdot)\in\Mcl$ and $f^{[m]}_n(\cdot)\in\Mcl$; in particular both are
      		nonnegative, nonincreasing and $1$-Lipschitz;
	\item $\Delta V^{[m+1]}_n(x)\le\Delta V^{[m]}_n(x)$ and
      		$f^{[m+1]}_n(x)\le f^{[m]}_n(x)$ (concavity in the number of rights);
	\item $\Delta V^{[m]}_n(x)\le\Delta V^{[m]}_{n+1}(x)$ and
      		$f^{[m]}_n(x)\le f^{[m]}_{n+1}(x)$ (monotonicity in the remaining time);
	\item $V^{[m]}_n(x)$ is convex, nonincreasing and $\mu$-Lipschitz with
      		$\mu:=\min\{m,n+1\}$, and $V^{[m]}_n(x)$ is nondecreasing in $n$ and in $m$.
\end{enumerate}
\end{proposition}
%%%%%%%%%%%%%%%%%%%%%%%%%%%%%%%%%%%%%%%%%%%%%%%
\begin{proof}
\emph{(i) and (ii).}  We use induction on $n$.  
For $n=0$, $\Delta V^{[1]}_0(x)=g(\cdot)\in\Mcl$, $
	\Delta V^{[m]}_0(x)=0$ for $m\ge2$, and
	$f^{[m]}_0(x)=0$ for all $m\ge1$, so both assertions hold.

Let $n\ge1$ and assume (i), (ii) at $n-1$.  
Then $f^{[m]}_n(x)=(\A \Delta V^{[m]}_{n-1})(\cdot)\in\Mcl$ by
	Lemma~\ref{lem:operator}(iii), and
	$f^{[m+1]}_n(x)\le f^{[m]}_n(x)$ by Lemma~\ref{lem:operator}(i).  
For $m=1$,
	$ \Delta V^{[1]}_n(x)=\max\{g(x),f^{[1]}_n(x)\}\in\Mcl .$
For $m\ge2$, the just established ordering
	$f^{[m]}_n(x)\le f^{[m-1]}_n(x)$ permits the use of Lemma~\ref{lem:median}, and
	\begin{align*}
   		\Delta V^{[m]}_n(x)=\med\{f^{[m]}_n(x),g(x),f^{[m-1]}_n(x)\}\in\Mcl .
	\end{align*}
It remains to propagate the ordering of the marginal values.  
For $m=1$,
	\begin{align*}
  		\Delta V^{[2]}_n(x)&=\min\{f^{[1]}_n(x),\max\{g(x),f^{[2]}_n(x)\}\} \\
  			&\le \max\{g(x),f^{[1]}_n(x)\}=\Delta V^{[1]}_n(x) .
	\end{align*}
For $m\ge2$, monotonicity of the interval projection in both endpoints gives
	\begin{align*}
 		 \Delta V^{[m+1]}_n(x)
		 	&=\min\big\{f^{[m]}_n(x),\max\{g(x),f^{[m+1]}_n(x)\}\big\} \\
  			&\le \min\big\{f^{[m-1]}_n(x),\max\{g(x),f^{[m]}_n(x)\}\big\}
  			=\Delta V^{[m]}_n(x) .
	\end{align*}
This proves (i) and (ii).

\emph{(iii).}  Again induct on $n$.  
At $n=0$,
$	\Delta V^{[1]}_0(x)=g(x)\le\max\{g(x),f^{[1]}_1(x)\}=\Delta V^{[1]}_1(x)$, while for
	$m\ge2$ the claim follows from nonnegativity. 
 Suppose $\Delta V^{[m]}_{n-1}(x)\le\Delta V^{[m]}_n(x)$ for every $m$.  
 Then
	\begin{align*}
  		f^{[m]}_n(x)=(\A \Delta V^{[m]}_{n-1})(x)
  			\le (\A \Delta V^{[m]}_n)(x)=f^{[m]}_{n+1}(x).
	\end{align*}
For $m=1$ this implies
	$\max\{g(x),f^{[1]}_n(x)\}\le\max\{g(x),f^{[1]}_{n+1}(x)\}$.
For $m\ge2$, the median is nondecreasing in each argument, hence
	\begin{align*}
  		\Delta V^{[m]}_n(x)&=\med\{f^{[m]}_n(x),g(x),f^{[m-1]}_n(x)\} \\
 			& \le\med\{f^{[m]}_{n+1}(x),g(x),f^{[m-1]}_{n+1}(x)\}
  				=\Delta V^{[m]}_{n+1}(x).
	\end{align*}

\emph{(iv).}  Convexity follows by induction from \eqref{eq:dp2}: $g(x)$ is convex,
	$\A$ preserves convexity, and the maximum of two convex functions is convex.
Monotonicity in $x$ is proved in the same way.  
Since $V^{[m]}_n(x)=\sum_{j=1}^{m}\Delta V^{[j]}_n(x)$ and each summand is
	$1$-Lipschitz by (i), $V^{[m]}_n(x)$ is $m$-Lipschitz.  
By Lemma~\ref{lem:saturation}, the summands with $j\ge n+2$ vanish, so the
	Lipschitz constant is at most $\mu=\min\{m,n+1\}$.  
Monotonicity in $n$ follows from (iii), and monotonicity in $m$ from the nonnegativity in (i).
\end{proof}

\begin{remark}\label{rem:noconvexity}
Proposition~\ref{prop:structure} does \emph{not} assert that
	$\Delta V^{[m]}_n(x)$ or $f^{[m]}_n(x)$ is convex, and indeed they are not; see
	Example~\ref{ex:nonconvex}.  
The point of the present formulation is that the
	interval structure of the exercise region, which for $m=1$ one obtains from
	convexity, is in fact a consequence of the weaker and stable property
	$(f^{[m]}_n(x))'\ge-1$.
\end{remark}

We next record the exact behaviour near $x=0$, which we shall need to locate the
	exercise boundary.  
Put
	\begin{equation*}
  		a_\mu:=\Big(\sum_{i=0}^{\mu-1}\alpha^{i}\Big)K,\qquad \mu\ge1,\qquad a_0:=0 .
	\end{equation*}
%%%%%%%%%%%%%%%%%%%%%%%%%%%%%%%%%%%%%%%%%%%%%%%
\begin{lemma}\label{lem:deep}
Let $n\ge0$, $m\ge1$ and $\mu=\min\{m,n+1\}$. 
For every $x\in(0,K\lambda^{-n})$,
	\begin{equation}\label{eq:deep}
  		V^{[m]}_n(x)=a_\mu-\mu x .
	\end{equation}
In particular $\lim_{x\downarrow0}V^{[m]}_n(x)=a_\mu$ and the slope of
	$V^{[m]}_n(x)$ deep inside the exercise region is exactly $-\mu$.  
Moreover, for $m\le n$ and $x\in(0,K\lambda^{-n})$,
	\begin{equation}\label{eq:deepf}
  		\Delta V^{[m]}_n(x)=\alpha^{m-1}K-x,\qquad f^{[m]}_n(x)=\alpha^{m}K-x ,
	\end{equation}
	and $f^{[m]}_n(x)\equiv0$ for $m\ge n+1$.
\end{lemma}
%%%%%%%%%%%%%%%%%%%%%%%%%%%%%%%%%%%%%%%%%%%%%%%
\begin{proof}
Induction on $n$.  For $n=0$ and $x<K$, $V^{[m]}_0(x)=K-x=a_1-1\cdot x$ and
	$\mu=1$.  Let $n\ge1$ and $x<K\lambda^{-n}$.  
Then $\lambda x<K\lambda^{-(n-1)}$ and $\lambda^{-1}x<K\lambda^{-(n-1)}$, 
	so the induction hypothesis applies at $\lambda^{\pm1}x$ and, 
	by Lemma~\ref{lem:operator}(v),
	$\A V^{[j]}_{n-1}(x)=\alpha a_{\mu_j}-\mu_j x$ with $\mu_j=\min\{j,n\}$.
Consider \eqref{eq:dp2}.  If $m\le n$ then $\mu_{m-1}=m-1$, $\mu_m=m$, and the
	exercise value is $K-x+\alpha a_{m-1}-(m-1)x=a_m-mx$ (using
	$K+\alpha a_{m-1}=a_m$) while the continuation value is $\alpha a_m-mx<a_m-mx$;
	hence \eqref{eq:deep} with $\mu=m$.  
If $m\ge n+1$ then
	$\mu_{m-1}=\mu_m=n$, the exercise value is $K-x+\alpha a_n-nx=a_{n+1}-(n+1)x$
	and the continuation value is $\alpha a_n-nx$, whose difference is $K-x>0$;
	hence \eqref{eq:deep} with $\mu=n+1$.  
Formulae \eqref{eq:deepf} follow by
	subtracting \eqref{eq:deep} at levels $m$ and $m-1$ and applying
	Lemma~\ref{lem:operator}(v); the last claim is Lemma~\ref{lem:saturation}.
\end{proof}

%%%%%%%%%%%%%%%%%%%%%%%%%%%%%%%%%%%%%%%%%%%%%%%
\subsection{The optimal exercise rule}
%%%%%%%%%%%%%%%%%%%%%%%%%%%%%%%%%%%%%%%%%%%%%%%
The following theorem is the principal exercise result for the American put.
%%%%%%%%%%%%%%%%%%%%%%%%%%%%%%%%%%%%%%%%%%%%%%%
\begin{theorem}%[American put: optimal threshold policy]
\label{thm:main}
Assume $r>0$.  For $n\ge0$ and $m\ge1$ define the threshold
	\begin{equation*}
  		x^{[m]*}_n:=\sup\big\{x\in(0,K]:\ g(x)\ge f^{[m]}_n(x)\big\}
	\end{equation*}
	and the exercise set
	\begin{equation}\label{eq:selectedD}
  		D^{[m]}_n:=\big\{x\in(0,K]:\ g(x)\ge f^{[m]}_n(x)\big\}.
	\end{equation}
Then:
\begin{enumerate}[(i)]
	\item the map $x\mapsto g(x)-f^{[m]}_n(x)$ is nonincreasing on $(0,K]$, is
      		positive near $0$ and nonpositive at $K$; consequently
     		 \begin{align*}
		 	x^{[m]*}_n\in(0,K],\qquad D^{[m]}_n=(0,x^{[m]*}_n]. 
		\end{align*}
     		 Exercise is an optimal action on $D^{[m]}_n$.  
		 For $x>K$ we select continuation; when $g(x)=f^{[m]}_n(x)=0$, 
		 this is merely a tie-breaking convention between two optimal actions.
	\item
      		\begin{align*}
        		0<x^{[m]*}_N\le x^{[m]*}_{N-1}\le\cdots
        		\le x^{[m]*}_1\le x^{[m]*}_0=K,
		\end{align*}
      		and $x^{[m]*}_n=K$ for all $n\le m-1$.
	\item $x^{[1]*}_n\le x^{[2]*}_n\le\dots\le x^{[m]*}_n$ for every $n$, and hence
      		$D^{[1]}_n\subseteq D^{[2]}_n\subseteq\dots\subseteq D^{[m]}_n$.
	\item Define recursively
      		\begin{align}
       		 	\sigma^{*}_m&=\min\big\{\nu\in\{0,\dots,N\}:\ 
          		 S_\nu\in D^{[m]}_{N-\nu}\big\},\label{eq:sigma1}\\
        		\sigma^{*}_{i}&=\min\big\{\nu\in\{0,\dots,N\}:\ 
          	 	\nu>\sigma^{*}_{i+1},\ S_\nu\in D^{[i]}_{N-\nu}\big\},
           		\qquad i=m-1,\dots,1,\label{eq:sigma2}
     		 \end{align}
     		 with $\min\emptyset:=\partial$.  
		 The finite entries of
      		$(\sigma_m^*,\dots,\sigma_1^*)$ are strictly increasing and all unused
      		rights are placed at $\partial$. 
		This policy is optimal, and
      		\begin{align*}
        		V^{[m]}_N(S_0)=\E\Big[\sum_{i:\,\sigma_i^*\le N}
          			\alpha^{\sigma_i^*}\big(K-S_{\sigma_i^*}\big)^{+}\Big]
        		=\sup_{\vec\tau\in\T^{[m]}_0}\E\Big[\sum_{i:\,\tau_i\le N}
          			\alpha^{\tau_i}\big(K-S_{\tau_i}\big)^{+}\Big].
		\end{align*}
\end{enumerate}
\end{theorem}
%%%%%%%%%%%%%%%%%%%%%%%%%%%%%%%%%%%%%%%%%%%%%%%
\begin{proof}
(i)  On $(0,K]$, $g(x)=K-x$.  For $0<y<x\le K$,
	\begin{align*}
   		g(x)-f^{[m]}_n(x)-g(y)+f^{[m]}_n(y)
   		=-(x-y)-\big(f^{[m]}_n(x)-f^{[m]}_n(y)\big)\le0,
	\end{align*}
	because $f^{[m]}_n(x)$ is $1$-Lipschitz.  
Thus $g(x)-f^{[m]}_n(x)$ is nonincreasing.
If $m\le n$, Lemma~\ref{lem:deep} gives
	$g(x)-f^{[m]}_n(x)=(1-\alpha^m)K>0$ for $x<K\lambda^{-n}$; if
	$m\ge n+1$, Lemma~\ref{lem:saturation} gives $f^{[m]}_n(x)\equiv0$, so the same
	difference is positive on $(0,K)$.  
At $K$ it equals $-f^{[m]}_n(K)\le0$.
Hence \eqref{eq:selectedD} is exactly $(0,x_n^{[m]*}]$.

For $x>K$, $g(x)=0$.  
If $f^{[m]}_n(x)>0$, continuation is strictly better;
	if $f^{[m]}_n(x)=0$, exercise and continuation have the same value.
Our selected policy chooses continuation in the latter case.
Thus zero-payoff tie points are excluded from the exercise region.

(ii) Proposition~\ref{prop:structure}(iii) gives
	$f^{[m]}_{n+1}(x)\ge f^{[m]}_n(x)$, hence
	$D^{[m]}_{n+1}\subseteq D^{[m]}_n$ and
	$x^{[m]*}_{n+1}\le x^{[m]*}_n$.  
Since $f^{[m]}_0(x)=0$,$x^{[m]*}_0=K$.  If $n\le m-1$, 
	Lemma~\ref{lem:saturation} gives
	$f^{[m]}_n(x)\equiv0$, and therefore $x^{[m]*}_n=K$.

(iii) Proposition~\ref{prop:structure}(ii) gives
	$f^{[m+1]}_n(x)\le f^{[m]}_n(x)$, so
	$D^{[m]}_n\subseteq D^{[m+1]}_n$.

(iv) At every state the rule \eqref{eq:sigma1}--\eqref{eq:sigma2} selects an
	action attaining the maximum in \eqref{eq:dp2}: it exercises on
	$D^{[i]}_{N-\nu}$, continues when continuation is strictly better, and also
	continues at the zero-payoff ties described in (i).  
Backward induction, equivalently Theorem~\ref{thm:multiple} with this optimal tie-breaking selector,
	therefore yields optimality of the recursively defined policy.  
The displayed value formula is just the corresponding discounted payoff, 
	with non-exercise-time entries contributing zero.
\end{proof}

The next two figures combine a schematic presentation with curves computed
	from the exact dynamic programming equation.  
We use $K=1$, $\lambda=1.2$ and $r=0.05$.  
They also display the local feature from
	Lemma~\ref{lem:deep}: for the $m=2$ continuation premium,
	$f_n^{[2]}(x)=\alpha^2K-x$ sufficiently close to zero, so the slope there is $-1$.
%%%%%%%%%%%%%%%%%%%%%%%%%%%%%%%%%%%%%%%%%%%%%%%
\begin{figure}[htbp]
\centering
\begin{minipage}{0.48\textwidth}
  \centering
  	\includegraphics[width=\linewidth]{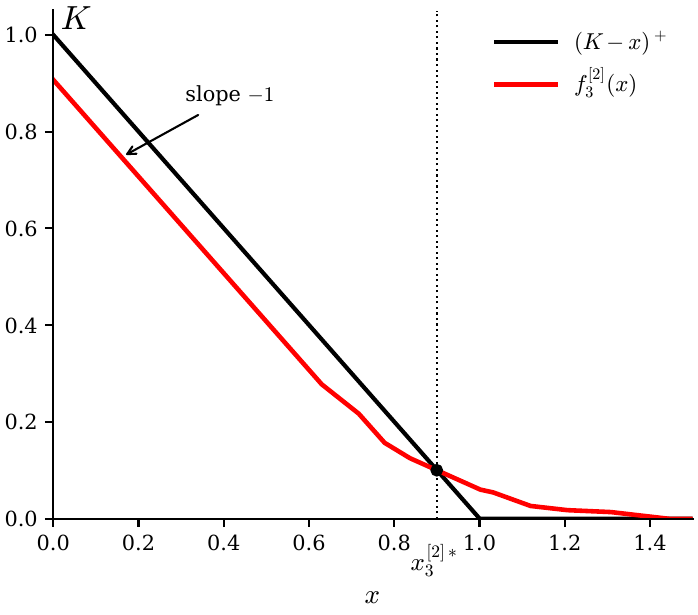}
 	 \caption{The single crossing of $g(x)$ and $f^{[2]}_3(x)$.  
	 	The vertical dotted line is $x^{[2]*}_3$.  
		The curve is obtained from the lattice recursion and has slope $-1$ near the origin.}
  \label{fig:fm-crossing}
\end{minipage}\hfill
\begin{minipage}{0.48\textwidth}
  \centering
  	\includegraphics[width=\linewidth]{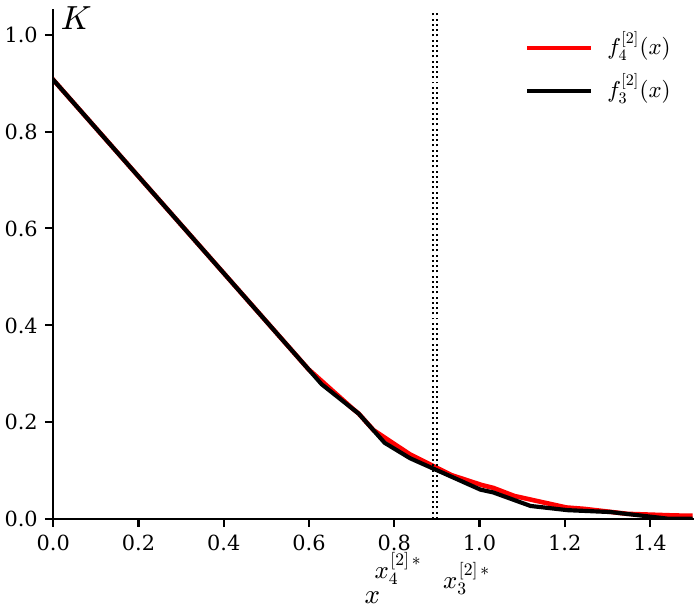}
 	 \caption{The comparison $f^{[2]}_4(x)\ge f^{[2]}_3(x)$, illustrating
 		 $x^{[2]*}_4\le x^{[2]*}_3$.  
		 The two continuation premia coincide with the same slope $-1$ affine branch near the origin.}
  \label{fig:fm-time}
\end{minipage}
\end{figure}
%%%%%%%%%%%%%%%%%%%%%%%%%%%%%%%%%%%%%%%%%%%%%%%
\begin{corollary}\label{cor:nested}
With $m$ rights the finite exercise times generated by
	Theorem~\ref{thm:main}(iv) are strictly increasing; any unused rights are placed
	at $\partial$. 
 By Theorem~\ref{thm:main}(iii), the boundary is lower when fewer
	rights remain.  
\end{corollary}
%%%%%%%%%%%%%%%%%%%%%%%%%%%%%%%%%%%%%%%%%%%%%%%
\begin{example}[$\Delta V^{[m]}_n(x)$ need not be convex]\label{ex:nonconvex}
Take $n=1$, $m=2$.  
Since $f^{[2]}_1(x)=0$ by Lemma~\ref{lem:saturation},
	Lemma~\ref{lem:median} gives
	\begin{align*}
  		 \Delta V^{[2]}_1(x)=\med\{0,\,g(x),\,f^{[1]}_1(x)\}=\min\big\{g(x),\ f^{[1]}_1(x)\big\},
   		\qquad f^{[1]}_1(x)=(\A g)(x) .
	\end{align*}
Explicitly, by \eqref{eq:operator} and \eqref{eq:gain},
	\begin{align*}
  	(\A g)(x)=
  		\begin{cases}
    			\alpha K-x, & 0<x\le K\lambda^{-1},\\[2pt]
    			\alpha q\,(K-\lambda^{-1}x), & K\lambda^{-1}\le x\le \lambda K,\\[2pt]
    			0,& x\ge\lambda K .
  		\end{cases}
	\end{align*}
Hence the slope of $\Delta V^{[2]}_1(x)=\min\{g(x),(\A g)(x)\}$ equals $-1$ on
	$(0,K\lambda^{-1})$, then $-\alpha q\lambda^{-1}$ on $(K\lambda^{-1},b)$, then
	$-1$ again on $(b,K)$, where
	\begin{equation*}
 	 	b=\frac{1-\alpha q}{1-\alpha q\lambda^{-1}}\,K
	\end{equation*}
	is the solution of $K-x=\alpha q(K-\lambda^{-1}x)$. 
The sequence of slopes $-1,\,-\alpha q\lambda^{-1},\,-1$ is not nondecreasing, 
 	so $\Delta V^{[2]}_1(x)$ is {not} convex.  
For $\lambda=1.2$, $r=0.05$, $K=1$ one finds
	$p=0.5909$, $q=0.4091$, $\alpha q=0.3896$, $b=0.9038$, 
	and the three slopes are $-1$, $-0.3247$, $-1$.
\end{example}

The same computation shows that $f^{[m]}_n(x)=(\A \Delta V^{[m]}_{n-1})(x)$ is in general
	not convex either, which is why the uniqueness of the exercise boundary cannot
	be obtained from a convexity/concavity single-crossing argument, and is obtained
	instead from the Lipschitz bound in the proof of Theorem~\ref{thm:main}(i).

%%%%%%%%%%%%%%%%%%%%%%%%%%%%%%%%%%%%%%%%%%%%%%%
\subsection{The free boundary: one-sided derivatives and the failure of smooth fit}
\label{sec:freeboundary}
%%%%%%%%%%%%%%%%%%%%%%%%%%%%%%%%%%%%%%%%%%%%%%%
In continuous time the value function of an American put is $C^1$ across the
	exercise boundary; this is the classical smooth fit (or smooth pasting) principle.  
On a fixed lattice this is \emph{false}.  
We make this precise, since the point is easy to get wrong and since the correct statement is a
	genuine structural difference between the discrete and the continuous model.
%%%%%%%%%%%%%%%%%%%%%%%%%%%%%%%%%%%%%%%%%%%%%%%
\begin{lemma}\label{lem:pwaffine}
For every $n\ge0$ and $m\ge1$ the function $V^{[m]}_n(x)$ is continuous, convex and
	piecewise affine with finitely many breakpoints on $(0,\infty)$.  
Consequently the one-sided derivatives $\partial_{-}V^{[m]}_n(x)$ and
	$\partial_{+}V^{[m]}_n(x)$ exist everywhere and
	$\partial_{-}V^{[m]}_n(x)\le\partial_{+}V^{[m]}_n(x)$.
\end{lemma}
%%%%%%%%%%%%%%%%%%%%%%%%%%%%%%%%%%%%%%%%%%%%%%%
\begin{proof}
$g(x)$ is continuous, convex and piecewise affine with one breakpoint, and $\A$
	maps this class into itself (a breakpoint of $\A\varphi$ lies at
	$\lambda^{\pm1}$ times a breakpoint of $\varphi$).  
The class is stable under maxima and sums, so \eqref{eq:dp2} propagates it.  
Convexity is Proposition~\ref{prop:structure}(iv), and a convex function has
	$\partial_{-}\le\partial_{+}$.
\end{proof}
%%%%%%%%%%%%%%%%%%%%%%%%%%%%%%%%%%%%%%%%%%%%%%%
\begin{proposition}\label{prop:onesided}
Let $r>0$, $n\ge1$, $m\ge1$, $\mu=\min\{m,n+1\}$ and $b:=x^{[m]*}_n$.  
If $b<K$ then
	\begin{align*}
   		-\mu\ \le\ \partial_{-}V^{[m]}_n(b)\ \le\ -1
   			\qquad\text{and}\qquad
  		 \partial_{-}V^{[m]}_n(b)\ \le\ \partial_{+}V^{[m]}_n(b)\ \le\ 0 .
	\end{align*}
Moreover $V^{[m]}_n(x)=a_\mu-\mu x$ for $x<K\lambda^{-n}$, so that the slope
	deep inside the exercise region equals $-\mu$ exactly.
\end{proposition}
%%%%%%%%%%%%%%%%%%%%%%%%%%%%%%%%%%%%%%%%%%%%%%%
\begin{proof}
The Lipschitz bound of Proposition~\ref{prop:structure}(iv) gives
	$\partial_{-}V^{[m]}_n(x)\ge-\mu$, and monotonicity gives
	$\partial_{+}V^{[m]}_n(x)\le0$.  
On $(0,b]$ we have $V^{[m]}_n(x)=g(x)+(\A V^{[m-1]}_{n-1})(x)$
	by \eqref{eq:dp3}, hence
	$\partial_{-}V^{[m]}_n(b)=-1+\partial_{-}(\A V^{[m-1]}_{n-1})(b)\le-1$,
	because $(\A V^{[m-1]}_{n-1})(x)$ is nonincreasing.  
The inequality $\partial_{-}\le\partial_{+}$ is Lemma~\ref{lem:pwaffine}, 
	and the last claim is Lemma~\ref{lem:deep}.
\end{proof}
%%%%%%%%%%%%%%%%%%%%%%%%%%%%%%%%%%%%%%%%%%%%%%%
\begin{proposition}%Smooth fit fails]
\label{ex:nosmoothfit}
Let $r>0$ and $\lambda>1+r$.  For $n=1$, $m=1$ the exercise boundary is
	\begin{align*}
   		x^{[1]*}_1=b=\frac{1-\alpha q}{1-\alpha q\lambda^{-1}}\,K
   		\ \in\ \big(K\lambda^{-1},\,K\big),
	\end{align*}
	and
	\begin{align*}
  	 	\partial_{-}V^{[1]}_1(b)=-1,
   		\qquad
   		\partial_{+}V^{[1]}_1(b)=-\frac{\alpha q}{\lambda}\ \in\ (-1,0) .
	\end{align*}
In particular $V^{[1]}_1(x)$ is not differentiable at the exercise boundary.
\end{proposition}
%%%%%%%%%%%%%%%%%%%%%%%%%%%%%%%%%%%%%%%%%%%%%%%
\begin{proof}
By \eqref{eq:dp3}, $V^{[1]}_1(x)=\max\{g(x),(\A g)(x)\}$ with $(\A g)(x)$ as computed in
	Example~\ref{ex:nonconvex}.  
The function $(\A g)(x)-g(x)=-(1-\alpha q)K+(1-\alpha q\lambda^{-1})x$ is affine and strictly
	increasing on $[K\lambda^{-1},K]$ and vanishes at $b$. 
It remains to check $K\lambda^{-1}<b<K$.  
The inequality $b<K$ is equivalent to
	$\alpha q\lambda^{-1}<\alpha q$, which holds since $\lambda>1$ and $q>0$.  
The inequality $b>K\lambda^{-1}$ is equivalent to
	$\alpha q<\lambda/(\lambda+1)$.  
Using $q=(\lambda-1-r)/(\lambda-\lambda^{-1})$
	and $\lambda-\lambda^{-1}=(\lambda-1)(\lambda+1)/\lambda$,
	\begin{align*}
   		\alpha q=\frac{\lambda(\lambda-1-r)}{(1+r)(\lambda-1)(\lambda+1)} ,
	\end{align*}
	so $\alpha q<\lambda/(\lambda+1)$ is equivalent to
	$\lambda-1-r<(1+r)(\lambda-1)$, i.e.\ to $0<r\lambda$, which holds because
	$r>0$.  
Hence $V^{[1]}_1(x)=g(x)$ on $(0,b]$, with slope $-1$, and
	$V^{[1]}_1(x)=(\A g)(x)=\alpha q(K-\lambda^{-1}x)$ on $[b,\lambda K]$, with slope
	$-\alpha q\lambda^{-1}$.
\end{proof}

For the numerical values used in Figures~\ref{fig:v1-kink} and
	\ref{fig:v1-two-times}, namely $\lambda=1.2$, $r=0.05$, $K=1$, the preceding
	proposition gives
	\begin{align*}
 		x^{[1]*}_1=0.903846\ldots,\qquad
 		\partial_-V^{[1]}_1(x^{[1]*}_1)=-1,\qquad
 		\partial_+V^{[1]}_1(x^{[1]*}_1)=-\alpha q/\lambda=-0.324675\ldots .
	\end{align*}
Figure~\ref{fig:kinkzoom} magnifies the resulting corner.
%%%%%%%%%%%%%%%%%%%%%%%%%%%%%%%%%%%%%%%%%%%%%%%
\begin{figure}[htbp]
  \centering
  \includegraphics[width=0.55\textwidth]{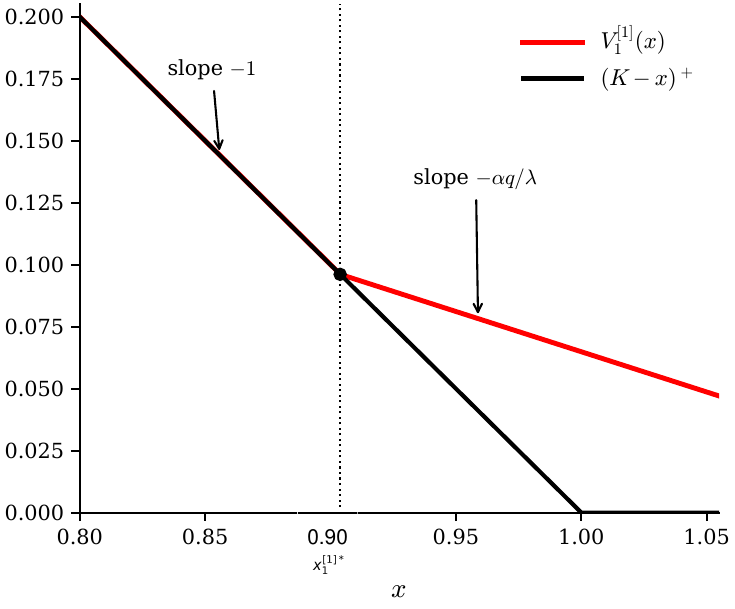}
 	 \caption{A magnified view of the fixed-mesh kink at the one-period exercise
 		 boundary.  
		 The arrows identify the two distinct one-sided slopes.  
		 This is an exact consequence of the one-step CRR recursion, not a smooth-fit drawing.}
  \label{fig:kinkzoom}
\end{figure}
%%%%%%%%%%%%%%%%%%%%%%%%%%%%%%%%%%%%%%%%%%%%%%%
Thus the correct discrete-time counterpart of the free boundary problem is
	\emph{continuous fit together with the variational inequality}, and not smooth fit.  
Explicitly, $(V^{[m]}_n(x))$ is characterised by
\begin{equation*}
  \left\{\begin{aligned}
    	&V^{[m]}_n(x)\ \ge\ g(x)+(\A V^{[m-1]}_{n-1})(x),\qquad
     		V^{[m]}_n(x)\ \ge\ (\A V^{[m]}_{n-1})(x),\\
    	&\big(V^{[m]}_n(x)-g(x)-(\A V^{[m-1]}_{n-1})(x)\big)\cdot
    	 	\big(V^{[m]}_n(x)-(\A V^{[m]}_{n-1})(x)\big)\ =\ 0,\\
    	&V^{[m]}_0(x)=g(x),\qquad V^{[0]}_n(x)\equiv0,
 	 \end{aligned}\right.
\end{equation*}
	together with the boundary conditions
	$V^{[m]}_n(x)=g(x)+(\A V^{[m-1]}_{n-1})(x)$ for $x\le x^{[m]*}_n$,
	$V^{[m]}_n(x)\to a_{\mu}$ as $x\downarrow0$ and $V^{[m]}_n(x)=0$ for
	$x\ge\lambda^{n}K$.  
No smooth-fit derivative condition is imposed, and none is valid in general.

Figures~\ref{fig:v1-kink}--\ref{fig:v1-two-times} use a clean schematic
	presentation while retaining value functions computed from the actual recursion.
In particular, the value curve leaves the payoff with a visible corner rather
	than tangentially.
%%%%%%%%%%%%%%%%%%%%%%%%%%%%%%%%%%%%%%%%%%%%%%%
\begin{figure}[htbp]
\centering
\begin{minipage}{0.48\textwidth}
  \centering
  	\includegraphics[width=\linewidth]{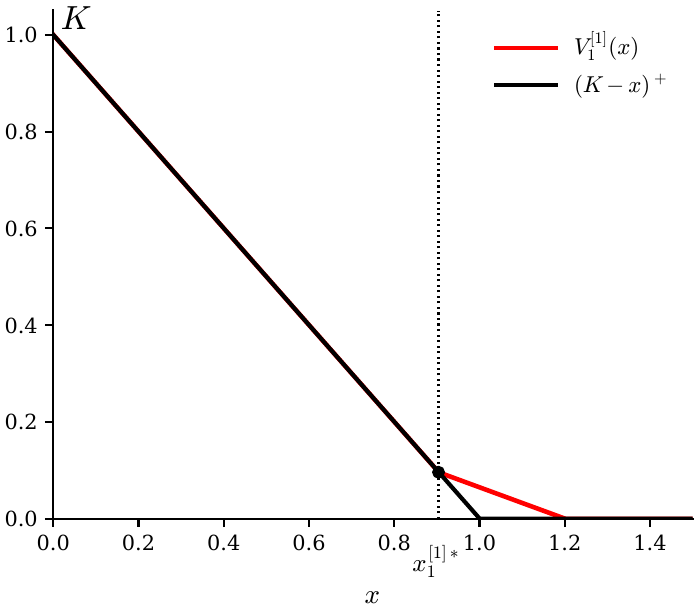}
 	 \caption{One remaining exercise right and one period to go.  
	 	The computed value $V^{[1]}_1(x)$ coincides with $g(x)$ up to 
		$x^{[1]*}_1$ and leaves it with a kink.}
	  \label{fig:v1-kink}
\end{minipage}\hfill
\begin{minipage}{0.48\textwidth}
  \centering
  	\includegraphics[width=\linewidth]{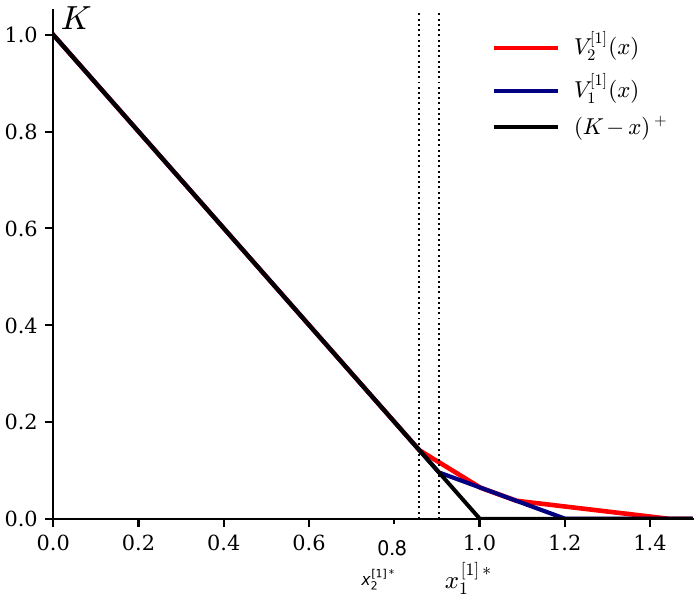}
  	\caption{Computed values $V^{[1]}_2(x)$ and $V^{[1]}_1(x)$.  
		Both boundaries are corners, and $x^{[1]*}_2<x^{[1]*}_1$, as required by
  		Theorem~\ref{thm:main}(ii).}
  	\label{fig:v1-two-times}
\end{minipage}
\end{figure}

%%%%%%%%%%%%%%%%%%%%%%%%%%%%%%%%%%%%%%%%%%%%%%%%%%%%%%%%%%%%%%%%%%%%%%%%%%%
\begin{remark}[Mesh refinement and the CRR limit]\label{rem:crr}
The fixed-mesh failure of smooth fit is compatible with smooth fit in the
	limiting Black--Scholes model.  
To illustrate this numerically, fix a horizon $T$, volatility $\sigma$ 
	and continuously compounded interest rate $\rho$, and
	use the standard Cox--Ross--Rubinstein scaling \cite{crr}
\begin{align*}
  	 \Delta t=T/N,\qquad
   	\lambda=e^{\sigma\sqrt{\Delta t}},\qquad
   	1+r=e^{\rho\Delta t},\qquad \alpha=e^{-\rho\Delta t}.
\end{align*}
For $m=1$, $K=1$, $T=0.3$, $\sigma=0.3$ and $\rho=0.05$, backward induction
	gives the initial exercise boundary $b$ and the one-sided derivatives shown below:
\begin{center}
	\begin{tabular}{rcccc}
	\hline
	$N$ & $\lambda$ & $b$ & $\partial_{-}V$ & $\partial_{+}V$\\
	\hline
	$4$   & $1.08563$ & $0.80052$ & $-1.0000$ & $-0.9129$\\
	$8$   & $1.05982$ & $0.80093$ & $-1.0000$ & $-0.9161$\\
	$16$  & $1.04193$ & $0.79148$ & $-1.0000$ & $-0.9471$\\
	$32$  & $1.02947$ & $0.78640$ & $-1.0000$ & $-0.9743$\\
	$64$  & $1.02075$ & $0.78361$ & $-1.0000$ & $-0.9784$\\
	$128$ & $1.01463$ & $0.78081$ & $-1.0000$ & $-0.9866$\\
	\hline
	\end{tabular}
\end{center}
For this sequence of meshes the derivative gap decreases markedly as $N$
	increases.  
The table is numerical evidence, not a proof of convergence of the
	derivatives.  
It is consistent with the classical smooth-fit property of the
	limiting Black--Scholes American put and with the usual CRR convergence of
	option values.
\end{remark}
%%%%%%%%%%%%%%%%%%%%%%%%%%%%%%%%%%%%%%%%%%%%%%%%%%%%%%%%%%%%%%%%%%%%%%%%%%%

For comparison with the continuous-time limit, Figure~\ref{fig:bs-three-boundaries}
	records a Black--Scholes benchmark with a genuine refractory period.
This computation is not used in any of the discrete-time arguments above.
Under
	\begin{align*}
  		dS_t=rS_t\,dt+\sigma S_t\,dB_t,
	\end{align*}
	we take
  		$K=100,\ T=1,\ r=0.05,\ \sigma=0.30,\ \delta=0.10, $
	and allow at most three exercises, with consecutive exercises separated by at
	least $\delta$. 
The coupled variational inequalities are solved numerically in Ano~\cite{ano} by
	implicit Euler finite differences and PSOR, while the refraction expectation is
	evaluated by Gauss--Hermite quadrature.  
The resulting boundaries satisfy
	\begin{align*}
 		 b^{[1]}(t)\le b^{[2]}(t)\le b^{[3]}(t).
	\end{align*}
Moreover, the time constraint forces
	$b^{[3]}=b^{[2]}$ on $(T-2\delta,T-\delta]$ and
	$b^{[3]}=b^{[2]}=b^{[1]}$ on $(T-\delta,T]$; the corresponding deadline
	jumps are visible at $T-2\delta=0.8$ and $T-\delta=0.9$.
%%%%%%%%%%%%%%%%%%%%%%%%%%%%%%%%%%%%%%%%%%%%%%%%%%%%%%%%%%%%%%%%%%%%%%%%%%%
\begin{figure}[htbp]
  \centering
  \includegraphics[width=0.75\textwidth]{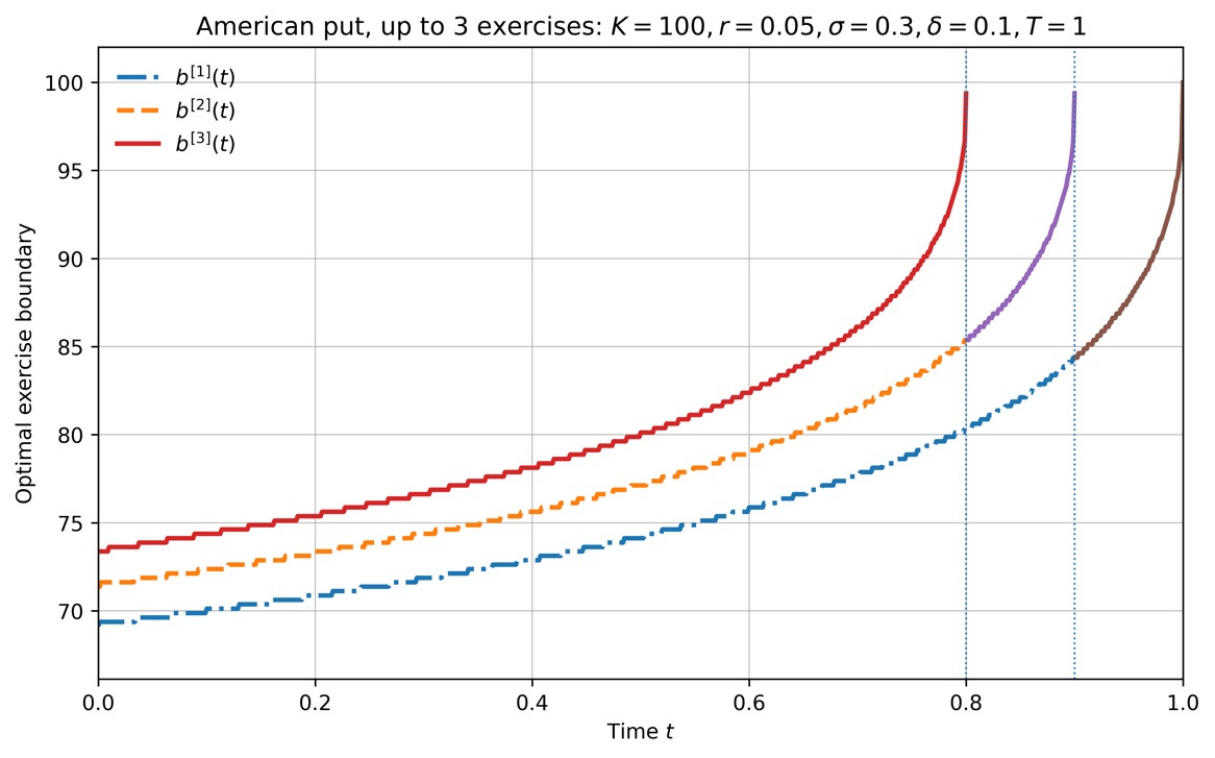}
 	 \caption{Continuous-time Black--Scholes benchmark for an American put with
  		up to three exercise rights and refractory period $\delta=0.10$.
  		Parameters are $K=100$, $T=1$, $r=0.05$ and $\sigma=0.30$.
  		The computed initial boundaries are
  		$b^{[1]}(0)\simeq69.13$, $b^{[2]}(0)\simeq71.38$ and
 		 $b^{[3]}(0)\simeq73.38$.}
  \label{fig:bs-three-boundaries}
\end{figure}

%%%%%%%%%%%%%%%%%%%%%%%%%%%%%%%%%%%%%%%%%%%%%%%
\section{American put with  random maturity}\label{sec:random}
%%%%%%%%%%%%%%%%%%%%%%%%%%%%%%%%%%%%%%%%%%%%%%%
\subsection{Set-up}
%%%%%%%%%%%%%%%%%%%%%%%%%%%%%%%%%%%%%%%%%%%%%%%
Let $\widetilde N$ be a random variable with values in $\{0,1,\dots,N\}$,
	independent of the price process, modelling a maturity which is not known in advance.  
We assume $\PP(\widetilde N=N)>0$, so every conditional survival
	probability used below is well defined.  
The option is void from calendar time $\widetilde N+1$ onwards.  
The holder does not observe $\widetilde N$ before it occurs, so exercise times are
	stopping times of the price filtration only, and the problem with $m$ rights is
	\begin{equation}\label{eq:randomobj}
   		\sup_{\vec\sigma\in\T^{[m]}_0}\ \E\Big[\sum_{i:\,\sigma_i\le N}
   			\alpha^{\sigma_i}\mathbf{1}_{\{\sigma_i\le\widetilde N\}}
  		 	\big(K-S_{\sigma_i}\big)^{+}\Big]
   			\;=\;\sup_{\vec\sigma\in\T^{[m]}_0}\
   			\E\Big[\sum_{i:\,\sigma_i\le N}\alpha^{\sigma_i}
   			\PP\big(\widetilde N\ge\sigma_i\big)\big(K-S_{\sigma_i}\big)^{+}\Big],
	\end{equation}
	the equality following from independence.  	
At calendar time $\nu=N-n$, conditional on the option still being alive, set
	\begin{equation}\label{eq:pi}
  		\pi_{n-1}:=\PP\big(\widetilde N\ge \nu+1\mid \widetilde N\ge \nu\big)
  			=\PP\big(\widetilde N\ge N-n+1\mid \widetilde N\ge N-n\big),
  			\qquad n=1,\dots,N .
	\end{equation}
Thus $\pi_{n-1}$ is the one-step conditional survival probability.  
For an elapsed time $k\in\{0,\dots,n\}$ define
	\begin{equation}\label{eq:Pi}
  		\Pi_{n,k}:= \PP\big(\widetilde N\ge N-n+k\mid \widetilde N\ge N-n\big)
 		 =\prod_{j=0}^{k-1}\pi_{n-1-j}, \qquad \Pi_{n,0}:=1 .
	\end{equation}
Let $(S^{x}_k)_{k=0}^{n}$ denote the price process started from $x$ at elapsed
	ime $0$.  
Conditionally on survival to the current date, the value with $m$ rights is
	\begin{equation}\label{eq:randomvalue}
  		\bar V^{[m]}_n(x)=\sup_{\vec\theta}\E_x\Big[\sum_{i:\,\theta_i\le n}
       		\alpha^{\theta_i}\Pi_{n,\theta_i}\big(K-S^{x}_{\theta_i}\big)^{+}\Big],
	\end{equation}
	where $\vec\theta$ is an admissible vector of \emph{elapsed} stopping times for
	the forward filtration of $(S^x_k)$, with unused rights sent to the non-exercise time.  
This formulation avoids treating the decreasing remaining-time index as
	a stopping time.  The weight in \eqref{eq:randomvalue} is the cumulative
	survival probability $\Pi_{n,k}$, not merely the one-step probability $\pi$.
%%%%%%%%%%%%%%%%%%%%%%%%%%%%%%%%%%%%%%%%%%%%%%%
\begin{assumption}\label{ass:A2}
\textbf{(A2)}\quad $\pi_0\le\pi_1\le\dots\le\pi_{N-1}$, 
	i.e.\ $n\mapsto\pi_{n-1}$ is nondecreasing.
\end{assumption}
%%%%%%%%%%%%%%%%%%%%%%%%%%%%%%%%%%%%%%%%%%%%%%%
Assumption (A2) says that the conditional one-step survival probability is
	larger when more time remains, that is, that the hazard rate of $\widetilde N$
	is nondecreasing in calendar time; equivalently, the tail sums of the law of
	$\widetilde N$ form a log-concave sequence.

The dynamic programming equation corresponding to \eqref{eq:randomvalue} is
	$\bar V^{[0]}_n(x)\equiv0,$
	\begin{equation}\label{eq:dprandom}
 		\bar V^{[m]}_0(x)=g(x),\quad 
  		\bar V^{[m]}_n(x)=\max\big\{\,g(x)+(\A_n \bar V^{[m-1]}_{n-1})(x),\
                      (\A_n \bar V^{[m]}_{n-1})(x)\,\big\},\quad n\ge1,
	\end{equation}
	where the one-step operator now carries the survival factor,
	\begin{equation*}
	  	(\A_n\varphi)(x):=\alpha\,\pi_{n-1}\big[p\,\varphi(\lambda x)
   	 		q\,\varphi(\lambda^{-1}x)\big]=\pi_{n-1}\,(\A\varphi)(x).
	\end{equation*}
By \eqref{eq:gain} the gain of $\A_n$ is $\pi_{n-1}\le1$, so that
	$\A_n(\Mcl)\subset\Mcl$ and in fact $\A_n\varphi$ is $\pi_{n-1}L$-Lipschitz
	whenever $\varphi$ is $L$-Lipschitz.  
The same median argument can now be repeated with the time-dependent
	operators $\A_n$.  
Set
	\begin{equation*}
	  	\Delta\bar V^{[m]}_n(x):=\bar V^{[m]}_n(x)-\bar V^{[m-1]}_n(x),
 	 	\qquad
  		\bar f^{[m]}_n(x):=(\A_n \Delta\bar V^{[m]}_{n-1})(x)\ (n\ge1),
	\end{equation*}
 	$\bar f^{[m]}_0(x):=0,\ \bar f^{[0]}_n(x):\equiv+\infty$,
	so that
	\begin{equation*}
  		\bar V^{[m]}_n(x)=\max\big\{g(x),\ \bar f^{[m]}_n(x)\big\}+(\A_n \bar V^{[m-1]}_{n-1})(x) .
	\end{equation*}
%%%%%%%%%%%%%%%%%%%%%%%%%%%%%%%%%%%%%%%%%%%%%%%
\begin{lemma}\label{lem:medianrandom}
For $n\ge0$ and $m\ge2$, if $\bar f^{[m]}_n(x)\le \bar f^{[m-1]}_n(x)$ then
	$$\Delta\bar V^{[m]}_n(x)=\med\{\bar f^{[m]}_n(x),\,g(x), 
	\,\bar f^{[m-1]}_n(x)\}$$; for $m=1$,
	$\Delta\bar V^{[1]}_n(x)=\max\{g(x),\bar f^{[1]}_n(x)\}$.
\end{lemma}
%%%%%%%%%%%%%%%%%%%%%%%%%%%%%%%%%%%%%%%%%%%%%%%
\begin{proof}
Identical to Lemma~\ref{lem:median}, using
	$(\A_n \Delta\bar V^{[m-1]}_{n-1})(x)=\bar f^{[m-1]}_n(x)$.
\end{proof}
%%%%%%%%%%%%%%%%%%%%%%%%%%%%%%%%%%%%%%%%%%%%%%%
\begin{proposition}\label{prop:structurerandom}
For arbitrary one-step survival probabilities $\pi_j\in[0,1]$, and for all
	$n\ge0$ and $m\ge1$,
	\begin{align*}
  		 \Delta\bar V^{[m]}_n(\cdot)\in\Mcl,\quad \bar f^{[m]}_n(\cdot)\in\Mcl,
   		\quad
   		\Delta\bar V^{[m+1]}_n(x)\le\Delta\bar V^{[m]}_n(x),\quad
   		\bar f^{[m+1]}_n(x)\le\bar f^{[m]}_n(x) .
	\end{align*}
Moreover $\bar V^{[m]}_n(x)$ is convex, nonincreasing and
	$\min\{m,n+1\}$-Lipschitz, and
	\begin{equation}\label{eq:limrandom}
  		\lim_{x\downarrow0}\bar f^{[m]}_n(x)
  		=\alpha^{m}\Big(\prod_{j=1}^{m}\pi_{n-j}\Big)K \quad (m\le n),
  			\qquad \bar f^{[m]}_n(x)\equiv0\quad(m\ge n+1).
	\end{equation}
If, in addition, (A2) holds, then
	\begin{align*}
   		\Delta\bar V^{[m]}_n(x)\le\Delta\bar V^{[m]}_{n+1}(x),
   		\qquad
   		\bar f^{[m]}_n(x)\le\bar f^{[m]}_{n+1}(x).
	\end{align*}
\end{proposition}
%%%%%%%%%%%%%%%%%%%%%%%%%%%%%%%%%%%%%%%%%%%%%%%
\begin{proof}
The proof is the time-inhomogeneous analogue of
	Proposition~\ref{prop:structure}.  
At a fixed $n$, the induction that proves
	membership in $\Mcl$ and concavity in $m$ is unchanged, because the same
	operator $\A_n$ acts at every level $m$.  
In particular, after the ordering
	$\bar f^{[m+1]}_n(x)\le\bar f^{[m]}_n(x)$ is obtained from the induction hypothesis,
	Lemma~\ref{lem:medianrandom} applies and the interval-projection argument propagates both
	properties.

For monotonicity in $n$, assume
	$\Delta\bar V^{[m]}_{n-1}(x)\le\Delta\bar V^{[m]}_n(x)$ for every $m$.  
Under (A2),
	$\pi_n\ge\pi_{n-1}$, so positivity and monotonicity of $\A$ give
	\begin{align*}
  		\bar f^{[m]}_{n+1}(x)=\pi_n(\A \Delta\bar V^{[m]}_n)(x)
  		\ge \pi_{n-1}(\A \Delta\bar V^{[m]}_{n-1})(x)
  		=\bar f^{[m]}_n(x) .
	\end{align*}
The median identity (or the maximum formula when $m=1$) then yields
	$\Delta\bar V^{[m]}_n(x)\le\Delta\bar V^{[m]}_{n+1}(x)$.

Convexity and monotonicity of $\bar V^{[m]}_n(x)$ follow directly from
	\eqref{eq:dprandom}; the Lipschitz estimate follows by summing the marginal
	values and using saturation exactly as in Proposition~\ref{prop:structure}.
Finally, the saturation proof is purely combinatorial and is unchanged by the
	time dependence of $\A_n$.  
Near $x=0$, the affine map $a-bx$ is sent by $\A_n$ to
	$\alpha\pi_{n-1}a-\pi_{n-1}bx$, and induction gives \eqref{eq:limrandom}.
\end{proof}
%%%%%%%%%%%%%%%%%%%%%%%%%%%%%%%%%%%%%%%%%%%%%%%
\begin{theorem}%[American put with random maturity]
\label{thm:mainrandom}
Assume $r>0$.  
Define
	\begin{align*}
  		\bar x^{[m]*}_n:=\sup\{x\in(0,K]:g(x)\ge\bar f^{[m]}_n(x)\},
  		\qquad
  		\bar D^{[m]}_n:=\{x\in(0,K]:g(x)\ge\bar f^{[m]}_n(x)\}.
	\end{align*}
Then $\bar D^{[m]}_n=(0,\bar x^{[m]*}_n]$, and
	$\bar x^{[m+1]*}_n\ge\bar x^{[m]*}_n$ for every $n,m$. 
If (A2) also holds,then
	\begin{align*}
   		\bar x^{[m]*}_{n+1}\le \bar x^{[m]*}_n ,
	\end{align*}
	so the boundary is monotone in the remaining time as well.

Starting at calendar time $0$, define the scheduled exercise times recursively
	by the boundary rule of Theorem~\ref{thm:main}(iv), with
	$D^{[i]}_{N-\nu}$ replaced by $\bar D^{[i]}_{N-\nu}$.  
The schedule is optimal for \eqref{eq:randomobj}; an exercise produces a payoff only on
	$\{\sigma_i^*\le\widetilde N\}$, and if the random maturity occurs before the
	next scheduled exercise, all remaining rights expire.
\end{theorem}
%%%%%%%%%%%%%%%%%%%%%%%%%%%%%%%%%%%%%%%%%%%%%%%
\begin{proof}
The interval statement uses only that $\bar f^{[m]}_n(x)$ is nonincreasing and
	$1$-Lipschitz.  
Its positivity near zero follows from
	\begin{align*}
  		\lim_{x\downarrow0}\big(g(x)-\bar f^{[m]}_n(x)\big)
  		=\Big(1-\alpha^m\prod_{j=1}^{m}\pi_{n-j}\Big)K>0
	\end{align*}
	when $m\le n$, while for $m\ge n+1$ saturation gives
	$\bar f^{[m]}_n(x)\equiv0$.  
At $x=K$, $g(K)-\bar f^{[m]}_n(K)=-\bar f^{[m]}_n(K)\le0$.  
Hence the selected in-the-money contact set is exactly an interval.  
The nesting in $m$ follows from $\bar f^{[m+1]}_n(x)\le\bar f^{[m]}_n(x)$.  
Under (A2), Proposition~\ref{prop:structurerandom} gives
	$\bar f^{[m]}_{n+1}(x)\ge\bar f^{[m]}_n(x)$, which yields the stated monotonicity in $n$.  
Finally, at every live state the selected boundary action attains the
	maximum in \eqref{eq:dprandom}; the cumulative survival factors in
	\eqref{eq:randomvalue} are exactly generated by successive applications of
	$\A_n$.  Backward induction therefore proves optimality.
\end{proof}
%%%%%%%%%%%%%%%%%%%%%%%%%%%%%%%%%%%%%%%%%%%%%%%
\begin{corollary}%[Fixed versus random maturity]
\label{cor:comparison}
For every $n$ and $m$, $\bar V^{[m]}_n(x)\le V^{[m]}_n(x)$ and
	$x^{[m]*}_n\le\bar x^{[m]*}_n$; equivalently, the exercise region of the
	random-maturity option contains that of the fixed-maturity option.
\end{corollary}
%%%%%%%%%%%%%%%%%%%%%%%%%%%%%%%%%%%%%%%%%%%%%%%
\begin{proof}
We prove simultaneously by induction on $n$ that
	$\Delta\bar V^{[m]}_n(x)\le\Delta V^{[m]}_n(x)$ for every $m$.  
At $n=0$ the marginal values coincide.  
If the claim holds at $n-1$, then
	\begin{align*}
  		\bar f^{[m]}_n(x)=\pi_{n-1}(\A \Delta\bar V^{[m]}_{n-1})(x)
  		\le (\A \Delta V^{[m]}_{n-1})(x)=f^{[m]}_n(x) .
	\end{align*}
For $m=1$, the maximum representation gives
	$\Delta\bar V^{[1]}_n(x)\le\Delta V^{[1]}_n(x)$; for $m\ge2$, apply the monotonicity of
	the median in each argument.  
Summing the marginal inequalities gives
	$\bar V^{[m]}_n(x)\le V^{[m]}_n(x)$, and
	$\bar f^{[m]}_n(x)\le f^{[m]}_n(x)$ implies
	$D^{[m]}_n\subseteq\bar D^{[m]}_n$, hence
	$x^{[m]*}_n\le\bar x^{[m]*}_n$.
\end{proof}

%%%%%%%%%%%%%%%%%%%%%%%%%%%%%%%%%%%%%%%%%%%%%%%
\subsection{Examples of maturity distributions satisfying (A2)}
%%%%%%%%%%%%%%%%%%%%%%%%%%%%%%%%%%%%%%%%%%%%%%%-
\begin{example}[Uniform]\label{ex:uniform}
Let $\widetilde N$ be uniform on $\{0,1,\dots,N\}$.  
Then $\PP(\widetilde N\ge N-n)=(n+1)/(N+1)$, so
	\begin{align*}
   		\pi_{n-1}=\frac{n}{n+1}=1-\frac{1}{n+1},\qquad n=1,\dots,N ,
	\end{align*}
	which is increasing in $n$; (A2) holds.  In particular $\pi_0=1/2$.
\end{example}

\begin{example}[Truncated geometric]\label{ex:geometric}
Let $\PP(\widetilde N=k)=c\,u^{k}$, $k=0,\dots,N$, with $u=1-\rho\in(0,1)$ and
	$c=(1-u)/(1-u^{N+1})$.  
Writing $T_k:=\sum_{j=k}^{N}u^{j}
	=u^{k}(1-u^{N-k+1})/(1-u)$ we have $\pi_{n-1}=T_{N-n+1}/T_{N-n}$, and (A2) is
	equivalent to the log-concavity of $k\mapsto T_k$:
	\begin{align*}
   		T_k^{2}\ \ge\ T_{k-1}T_{k+1}.
	\end{align*}
Putting $v:=u^{N-k+1}$, this reads $(1-v)^{2}\ge(1-uv)(1-v u^{-1})$, i.e.\
	$-2v\ge-v(u+u^{-1})$, i.e.\ $u+u^{-1}\ge2$, which holds for every $u>0$ by the
	arithmetic--geometric mean inequality.  
Hence (A2) holds for every $\rho\in(0,1)$.
\end{example}

\begin{example}[Truncated Poisson]\label{ex:poisson}
Let $\PP(\widetilde N=k)=c\,\theta^{k}/k!$, $k=0,\dots,N$, $\theta>0$.  
The sequence $a_k=\theta^{k}/k!$ is log-concave, since
	$a_k^{2}/(a_{k-1}a_{k+1})=(k+1)/k\ge1$ for $k\ge1$.  
The convolution of two nonnegative log-concave sequences without internal zeros is log-concave
	\cite{stanley}; applying this to $a$ and to the sequence
	$(1,1,1,\dots)$, and reversing the index, the tail sums
	$T_k=\sum_{j=k}^{N}a_j$ form a log-concave sequence.  
As in Example~\ref{ex:geometric} this is exactly (A2).
\end{example}

Assumption (A2) is needed only for monotonicity in the remaining time,
	not for the threshold structure.

%%%%%%%%%%%%%%%%%%%%%%%%%%%%%%%%%%%%%%%%%%%%%%%
\section{Russian option}%: numeraire reduction and multiple exercise}
\label{sec:russian}
%%%%%%%%%%%%%%%%%%%%%%%%%%%%%%%%%%%%%%%%%%%%%%%
\subsection{Reduction by a change of numeraire}
%%%%%%%%%%%%%%%%%%%%%%%%%%%%%%%%%%%%%%%%%%%%%%%
Let $S_0=s_0>0$, let $m_0\ge s_0$ and put
	\begin{align*}
   		M_\nu:=\Big(\max_{0\le i\le\nu}S_i\Big)\vee m_0,\qquad \nu=0,\dots,N .
	\end{align*}
The Russian option with $m$ rights pays $M_{\sigma_i}$ at each finite exercise
	date.  
Using the non-exercise-time convention of Section~\ref{sec:general}, its value is
	\begin{equation}\label{eq:russianobj}
   		\sup_{\vec\sigma\in\T^{[m]}_0}
  		 \E\Big[\sum_{i:\,\sigma_i\le N}
     		 \alpha^{\sigma_i}M_{\sigma_i}\Big].
	\end{equation}
The reward depends on the pair $(S_\nu,M_\nu)$; the classical device of Shepp
	and Shiryaev \cite{shepp1,shepp2} reduces it to the one-dimensional ratio
	$X_\nu:=M_\nu/S_\nu\ge1$.  
In discrete time this reduction is \emph{not} a mere
	substitution --- one has $\E[\alpha^{\sigma}M_\sigma]=\E[\alpha^{\sigma}
	X_\sigma S_\sigma]\ne\E[\alpha^{\sigma}X_\sigma]$ --- but a change of numeraire,
	which we now carry out.
%%%%%%%%%%%%%%%%%%%%%%%%%%%%%%%%%%%%%%%%%%%%%%%
\begin{lemma}\label{lem:numeraire}
Let $Z_\nu:=\alpha^{\nu}S_\nu/s_0$.  
Then $(Z_\nu)_{\nu=0}^{N}$ is a strictly
	positive $\PP$-martingale with $Z_0=1$.  
Define the probability measure $\Pt$ on
	$\F_N$ by $\dd\Pt/\dd\PP:=Z_N$.  
Then:
	\begin{enumerate}[(i)]
	\item under $\Pt$ the increments $(\varepsilon_\nu)$ are i.i.d.\ with
      		\begin{align*}
        		\widetilde p:=\Pt(\varepsilon_\nu=+1)=\alpha\lambda p,\qquad
        		\widetilde q:=\Pt(\varepsilon_\nu=-1)=\alpha\lambda^{-1}q,\qquad
        		\widetilde p+\widetilde q=1 ;
		\end{align*}
	\item for every stopping time $\sigma$ with values in $\{0,\dots,N\}$,
      		\begin{align*}
        		\E\big[\alpha^{\sigma}M_\sigma\big]=s_0\,\Et\big[X_\sigma\big] ;
		\end{align*}
	\item $(X_\nu)$ is a $\Pt$-Markov chain on $[1,\infty)$ with
      		\begin{align*}
        		X_{\nu+1}=\begin{cases}
          			(\lambda^{-1}X_\nu)\vee1, & \text{with probability }\widetilde p,\\
          			\lambda X_\nu, & \text{with probability }\widetilde q,
        		\end{cases}
        		\qquad X_0=m_0/s_0 .
		\end{align*}
\end{enumerate}
\end{lemma}
%%%%%%%%%%%%%%%%%%%%%%%%%%%%%%%%%%%%%%%%%%%%%%%
\begin{proof}
$\E[S_{\nu+1}\mid\F_\nu]=(p\lambda+q\lambda^{-1})S_\nu=(1+r)S_\nu$ by
	\eqref{eq:pq}, so $(Z_\nu)$ is a positive martingale with $Z_0=1$ and $\Pt$ is a
	probability measure equivalent to $\PP$.  
(i) follows from $Z_{\nu+1}/Z_\nu=\alpha\lambda^{\varepsilon_{\nu+1}}$,
	which gives the stated one-step weights, and from $\alpha(p\lambda+q\lambda^{-1})=1$, 
	which is \eqref{eq:gain}.  
For (ii), $M_\sigma=X_\sigma S_\sigma$ and 
	$\alpha^{\sigma}S_\sigma=s_0Z_\sigma$, so, $\sigma$ being bounded and
	$X_\sigma$ being $\F_\sigma$-measurable,
	\begin{align*}
  		\E\big[\alpha^{\sigma}M_\sigma\big]=s_0\,\E\big[Z_\sigma X_\sigma\big]
  			=s_0\,\E\big[\E[Z_N\mid\F_\sigma]\,X_\sigma\big]
  			=s_0\,\E\big[Z_N X_\sigma\big]=s_0\,\Et\big[X_\sigma\big].
	\end{align*}
For (iii), $M_{\nu+1}=M_\nu\vee S_{\nu+1}$ gives
	$X_{\nu+1}=(M_\nu\vee S_{\nu+1})/S_{\nu+1}=(X_\nu S_\nu/S_{\nu+1})\vee1$, and
	$X_\nu\ge1$ makes the maximum superfluous in the down case.
\end{proof}

By Lemma~\ref{lem:numeraire}(ii) applied to each finite component and by
	linearity, \eqref{eq:russianobj} equals
	\begin{align*}
  		s_0\sup_{\vec\sigma\in\T^{[m]}_0}
  	\Et\Big[\sum_{i:\,\sigma_i\le N}X_{\sigma_i}\Big],
	\end{align*}
	an \emph{undiscounted} multiple stopping problem for the reflected random walk
	$X$ under $\Pt$.  
The discount has been absorbed into the dynamics, as the following shows.
%%%%%%%%%%%%%%%%%%%%%%%%%%%%%%%%%%%%%%%%%%%%%%%
\begin{lemma}\label{lem:Boperator}
Define, for $\varphi:[1,\infty)\to[0,\infty)$,
	\begin{align*}
  		(\B\varphi)(x):=\widetilde p\,\varphi\big((\lambda^{-1}x)\vee1\big)
    		+\widetilde q\,\varphi(\lambda x),\qquad x\ge1 .
	\end{align*}
Then $\B$ is order preserving and preserves nonnegativity.  
If $\varphi$ is nondecreasing, so is $\B\varphi$; if $\varphi$ is both nondecreasing and
	convex, then $\B\varphi$ is convex.  
If $\varphi$ is $L$-Lipschitz, then $\B\varphi$ is
	$\alpha L$-Lipschitz.  In particular $\B$ is a strict contraction on
	Lipschitz constants when $r>0$.
\end{lemma}
%%%%%%%%%%%%%%%%%%%%%%%%%%%%%%%%%%%%%%%%%%%%%%%
\begin{proof}
Order preservation and nonnegativity follow from the positive weights.
Both state maps $x\mapsto(\lambda^{-1}x)\vee1$ and $x\mapsto\lambda x$ are
	nondecreasing, so monotonicity is preserved.  
The first state map is convex and the second is affine.  
Hence, when $\varphi$ is nondecreasing and convex, both
	compositions are convex and so is their positive linear combination.  
Finally, the two state maps are $\lambda^{-1}$- and $\lambda$-Lipschitz, respectively,and
	\begin{align*}
   		\widetilde p\lambda^{-1}+\widetilde q\lambda=\alpha p+\alpha q=\alpha ,
	\end{align*}
	which proves the Lipschitz claim.
\end{proof}

%%%%%%%%%%%%%%%%%%%%%%%%%%%%%%%%%%%%%%%%%%%%%%%
\subsection{Dynamic programming and the exercise boundary}
%%%%%%%%%%%%%%%%%%%%%%%%%%%%%%%%%%%%%%%%%%%%%%%
Write $g(x):=x$ for the reward, let $n=N-\nu$ be the remaining time and let
$V^{[m]}_n(x)$ denote the value, in the reduced problem, with $m$ rights, $n$
	periods to go and $X=x$.  
Then $V^{[0]}_n(x)\equiv0$, 
	\begin{equation}\label{eq:dprussian}
 	 	V^{[m]}_0(x)=g(x)\ (m\ge1),\quad 
  		V^{[m]}_n(x)=\max\big\{\,g(x)+(\B V^{[m-1]}_{n-1})(x),\ (\B V^{[m]}_{n-1})(x)\,\big\},
  		\ n\ge1 ,
	\end{equation}
	and the value of the original problem \eqref{eq:russianobj} is 
	$s_0\,V^{[m]}_N(m_0/s_0)$.  
Set, in complete analogy with \eqref{eq:DeltaVf},
	\begin{align*}
 		 \Delta V^{[m]}_n(x):=V^{[m]}_n(x)-V^{[m-1]}_n(x),\qquad
 	 	f^{[m]}_n(x):=(\B \Delta V^{[m]}_{n-1})(x)\ (n\ge1),
	\end{align*}
	$f^{[m]}_0(x):=0,\ f^{[0]}_n(x):\equiv+\infty$,
	so that
	\begin{equation*}
 	 	V^{[m]}_n(x)=\max\big\{g(x),\ f^{[m]}_n(x)\big\}+(\B V^{[m-1]}_{n-1})(x),
	\end{equation*}
	and exercise is an optimal action exactly when $x\ge f^{[m]}_n(x)$.  
Let
	\begin{align*}
  		\Mcl^{+}:=\Big\{\varphi:[1,\infty)\to[0,\infty)\ \Big|\
  		\varphi\ \text{nondecreasing and }1\text{-Lipschitz}\Big\}.
	\end{align*}
%%%%%%%%%%%%%%%%%%%%%%%%%%%%%%%%%%%%%%%%%%%%%%%
\begin{proposition}%[Marginal-value structure for the Russian option]
\label{prop:russianstructure}
Assume $r>0$.  For all $n\ge0$ and $m\ge1$:
\begin{enumerate}[(i)]
	\item for $m\ge2$,
      		$\Delta V^{[m]}_n(x)=\med\{f^{[m]}_n(x),\,g(x),\,f^{[m-1]}_n(x)\},$
      		while $\Delta V^{[1]}_n(x)=\max\{g(x),f^{[1]}_n(x)\}$;
	\item $\Delta V^{[m]}_n(\cdot)\in\Mcl^{+}$, and $f^{[m]}_n(\cdot)\in\Mcl^{+}$ is moreover
      		$\alpha$-Lipschitz;
	\item $\Delta V^{[m+1]}_n(x)\le\Delta V^{[m]}_n(x)$ and $f^{[m+1]}_n(x)\le f^{[m]}_n(x)$;
	\item $\Delta V^{[m]}_n(x)\le\Delta V^{[m]}_{n+1}(x)$ and
      		$f^{[m]}_n(x)\le f^{[m]}_{n+1}(x)$;
	\item $V^{[m]}_n(x)$ is convex, nondecreasing and $L_m$-Lipschitz with
      		$L_m:=1+\alpha+\dots+\alpha^{m-1}$.
\end{enumerate}
\end{proposition}
%%%%%%%%%%%%%%%%%%%%%%%%%%%%%%%%%%%%%%%%%%%%%%%
\begin{proof}
We argue simultaneously by induction on $n$, as in Proposition~\ref{prop:structure}.  
At $n=0$,$\Delta V^{[1]}_0(x)=g(x)=x\in\Mcl^+$, 
	$\Delta V^{[m]}_0(x)=0$ for $m\ge2$, and $f^{[m]}_0(x)=0$.

Assume the assertions about $\Delta V_{n-1}^{[m]}(x)$ and their ordering in $m$.
Lemma~\ref{lem:Boperator} gives
	$f^{[m]}_n(x)=(\B \Delta V^{[m]}_{n-1})(\cdot)\in\Mcl^+$, with Lipschitz constant at most
	$\alpha$, and also $f^{[m+1]}_n(x)\le f^{[m]}_n(x)$.  
For $m=1$,$\Delta V^{[1]}_n(x)=\max\{g(x),f^{[1]}_n(x)\}\in\Mcl^+$.  
For $m\ge2$ the ordering of the $f$'s permits the median identity, 
	and closure of $\Mcl^+$ under the mediangives $\Delta V^{[m]}_n(\cdot)\in\Mcl^+$.  
The interval-projection comparison used in Proposition~\ref{prop:structure} then yields
	$\Delta V^{[m+1]}_n(x)\le\Delta V^{[m]}_n(x)$.  
This proves (i)--(iii).

For (iv), the base step is
	$ \Delta V^{[1]}_0(x)=g(x)\le\max\{g(x),f^{[1]}_1(x)\}=\Delta V^{[1]}_1(x),$
	while $\Delta V^{[m]}_0(x)=0\le\Delta V^{[m]}_1(x)$ for $m\ge2$.  
If $\Delta V^{[m]}_{n-1}(x)\le\Delta V^{[m]}_n(x)$ for every $m$, order preservation of $\B$
	gives $f^{[m]}_n(x)\le f^{[m]}_{n+1}(x)$.  
The maximum formula for $m=1$ and the median formula for $m\ge2$ then give
	$\Delta V^{[m]}_n(x)\le\Delta V^{[m]}_{n+1}(x)$.

For (v), convexity and monotonicity follow by induction from
	\eqref{eq:dprussian} and Lemma~\ref{lem:Boperator}.  
Moreover,
	\begin{align*}
  		\operatorname{Lip}_x\big(V^{[m]}_n(x)\big)
  		\le \max\Big\{ 1+\alpha\,\operatorname{Lip}_x\big(V^{[m-1]}_{n-1}(x)\big),
      		\alpha\,\operatorname{Lip}_x\big(V^{[m]}_{n-1}(x)\big)\Big\},
	\end{align*}
	and a double induction yields
	$\operatorname{Lip}_x\big(V^{[m]}_n(x)\big)\le L_m=1+\alpha+\cdots+\alpha^{m-1}$.
\end{proof}

%%%%%%%%%%%%%%%%%%%%%%%%%%%%%%%%%%%%%%%%%%%%%%%
\begin{theorem}%[Russian option: optimal threshold policy]
\label{thm:russian}
Assume $r>0$.  For $n\ge0$ and $m\ge1$ define
	\begin{align*}
  		x^{[m]*}_n:=\inf\big\{x\ge1:\ x\ge f^{[m]}_n(x)\big\}.
	\end{align*}
Then:
\begin{enumerate}[(i)]
	\item $x\mapsto x-f^{[m]}_n(x)$ is strictly increasing on $[1,\infty)$, with
      		every secant slope at least $1-\alpha>0$, and tends to $+\infty$; hence
      		$x^{[m]*}_n\in[1,\infty)$ is well defined and the exercise region is
      		\begin{align*}
       		 D^{[m]}_n=\big\{x\ge1:\ x\ge f^{[m]}_n(x)\big\}
        		=\big[x^{[m]*}_n,\ \infty\big) ;
		\end{align*}
	\item $1=x^{[m]*}_0\le x^{[m]*}_1\le\dots\le x^{[m]*}_N$;
	\item $x^{[m+1]*}_n\le x^{[m]*}_n$, so that
      		$D^{[1]}_n\subseteq D^{[2]}_n\subseteq\dots\subseteq D^{[m]}_n$;
	\item with $\min\emptyset:=\partial$, define
      		\begin{align*}
        		\sigma^{*}_m=\min\big\{\nu\in\{0,\dots,N\}:\ 
           		X_\nu\ge x^{[m]*}_{N-\nu}\big\},
		\end{align*}
      		and recursively, for $i=m-1,\dots,1$,
     		\begin{align*}
        		\sigma^{*}_{i}=\min\big\{\nu\in\{0,\dots,N\}:\ 
           		\nu>\sigma^{*}_{i+1},\ X_\nu\ge x^{[i]*}_{N-\nu}\big\}.
		\end{align*}
      		The finite entries are strictly increasing and unused rights are placed at $\partial$.  
		This rule is optimal for \eqref{eq:russianobj}, and the value
      			of the option equals $s_0V^{[m]}_N(m_0/s_0)$.
\end{enumerate}
\end{theorem}
%%%%%%%%%%%%%%%%%%%%%%%%%%%%%%%%%%%%%%%%%%%%%%%
\begin{proof}
(i) By Proposition~\ref{prop:russianstructure}(ii), $f^{[m]}_n(x)$ is
	$\alpha$-Lipschitz, so for $y<x$,
	$\big(x-f^{[m]}_n(x)\big)-\big(y-f^{[m]}_n(y)\big)\ge(1-\alpha)(x-y)>0$.  
Since $f^{[m]}_n(x)$ has at most linear growth of slope $\alpha<1$, 
	$x-f^{[m]}_n(x)\to+\infty$.  
A strictly increasing function has $\{x\ge f^{[m]}_n(x)\}$ equal to a half-line.

(ii) $f^{[m]}_0(x)=0$ gives $x^{[m]*}_0=1$; by
	Proposition~\ref{prop:russianstructure}(iv), $f^{[m]}_{n+1}(x)\ge f^{[m]}_n(x)$, hence
	$\{x\ge f^{[m]}_{n+1}(x)\}\subseteq\{x\ge f^{[m]}_n(x)\}$ and
	$x^{[m]*}_{n+1}\ge x^{[m]*}_n$.

(iii) By Proposition~\ref{prop:russianstructure}(iii),
	$f^{[m+1]}_n(x)\le f^{[m]}_n(x)$, hence
	$\{x\ge f^{[m]}_n(x)\}\subseteq\{x\ge f^{[m+1]}_n(x)\}$.

(iv) Combine Theorem~\ref{thm:multiple} (in the Markovian form
	\eqref{eq:calendar}, applied under $\Pt$ to the chain $X$, whose reward
	$g(x)=x$ satisfies $\Et[\max_{\nu\le N}X_\nu]\le\lambda^{N}X_0<\infty$) with
	Lemma~\ref{lem:numeraire}(ii).  	
Since $\Pt\sim\PP$, the class of admissible 
	stopping vectors is the same under both measures.
\end{proof}

\begin{remark}
For an interior boundary $b=x^{[m]*}_n>1$, the same one-sided argument as in
Proposition~\ref{prop:onesided} gives
	\begin{align*}
   		1\le\partial_{+}V^{[m]}_n(b)\le L_m,\qquad
   		\partial_{-}V^{[m]}_n(b)\le\partial_{+}V^{[m]}_n(b).
	\end{align*}
The inequality between the one-sided derivatives can be strict, so smooth fit
	is not a general fixed-mesh property here either.  
Unlike the put, however,
	existence and uniqueness of the economically relevant boundary follow
	immediately from the strict monotonicity in Theorem~\ref{thm:russian}(i): the
	contraction property of $\B$ does the work that the unit Lipschitz bound does
	in Theorem~\ref{thm:main}.
\end{remark}

%%%%%%%%%%%%%%%%%%%%%%%%%%%%%%%%%%%%%%%%%%%%%%%
\subsection{Random maturity}\label{sec:russian-random}
%%%%%%%%%%%%%%%%%%%%%%%%%%%%%%%%%%%%%%%%%%%%%%%
We now combine the share-numeraire reduction above with the independent random
	maturity of Section~\ref{sec:random}.  
Let $\widetilde N$ take values in $\{0,1,\dots,N\}$, 
	be independent of the stock-price process under $\PP$, and
	satisfy $\PP(\widetilde N=N)>0$.  
The contract is void from calendar time $\widetilde N+1$ onward.  
The holder does not observe $\widetilde N$ before it
	occurs, so the scheduled exercise times are stopping times of the price filtration.  
With $m$ rights the random-maturity Russian option has value
	\begin{equation}\label{eq:russianrandomobj}
  		\sup_{\vec\sigma\in\T^{[m]}_0}\E\Big[\sum_{i:\,\sigma_i\le N}
     		\alpha^{\sigma_i}\mathbf 1_{\{\sigma_i\le\widetilde N\}}
		M_{\sigma_i}\Big].
	\end{equation}

The same change of numeraire remains available.  
Extend the measure $\Pt$ of Lemma~\ref{lem:numeraire} from $\F_N$ to
	$\mathcal{G}_N:=\F_N\vee\sigma(\widetilde N)$ by the density
	$\dd\Pt/\dd\PP=Z_N$.  
Since $Z_N$ is measurable with respect to the price
	path only, $\widetilde N$ has the same law under $\Pt$ as under $\PP$ and is
	still independent of the price path.  Indeed, for $A\in\F_N$ and
	$B\in\sigma(\widetilde N)$,
	\begin{align*}
  		 \Pt(A\cap B)=\E[Z_N\mathbf 1_A\mathbf 1_B]
   			=\E[Z_N\mathbf 1_A]\,\PP(B)=\Pt(A)\,\Pt(B).
	\end{align*}
Consequently, for every price-filtration stopping time $\sigma\le N$,
\begin{align}
  	\E\big[\alpha^\sigma M_\sigma
      		\mathbf 1_{\{\sigma\le\widetilde N\}}\big]
  	&=s_0\sum_{j=0}^{N}
      		\E\big[Z_jX_j\mathbf 1_{\{\sigma=j\}}\big]
     		 \PP(\widetilde N\ge j)\notag\\
  	&=s_0\sum_{j=0}^{N}
      		\E\big[Z_NX_j\mathbf 1_{\{\sigma=j\}}\big]
      		\PP(\widetilde N\ge j)\notag\\
  	&=s_0\,\Et\big[X_\sigma
     		 \mathbf 1_{\{\sigma\le\widetilde N\}}\big].
  \label{eq:russianrandomnumeraire}
\end{align}
Here the second equality uses $\E[Z_N\mid\F_j]=Z_j$, while the first and
	last use independence of $\widetilde N$ from the price path under the
	corresponding measure.  
Thus random maturity does not interfere with the
	one-dimensional reduction: it only kills future rewards.

Let the one-step conditional survival probabilities $\pi_{n-1}$ and the
	cumulative factors $\Pi_{n,k}$ be as in \eqref{eq:pi}--\eqref{eq:Pi}.  
At a calendar date $\nu=N-n$, conditional on the contract still being alive and on
	$X_\nu=x$, let $(X^x_k)_{k=0}^{n}$ denote the reflected chain under $\Pt$
	started from $x$.  
Define the reduced value
	\begin{equation}\label{eq:russianrandomvalue}
  		U^{[m]}_n(x):=\sup_{\vec\theta}
     	\Et_x\Big[\sum_{i:\,\theta_i\le n}\Pi_{n,\theta_i}\,X^x_{\theta_i}\Big],
	\end{equation}
	where $\vec\theta$ ranges over admissible vectors of elapsed stopping times,
	with unused rights sent to the non-exercise time. 
 By \eqref{eq:russianrandomnumeraire}, the value of
	\eqref{eq:russianrandomobj} is
	\begin{align*}
   		s_0\,U_N^{[m]}(m_0/s_0).
	\end{align*}

For $n\ge1$ define the survival-weighted one-step operator
	\begin{equation*}
   		\B_n:=\pi_{n-1}\B,\qquad
   		(\B_n\varphi)(x)=\pi_{n-1}\Big[ \widetilde p\,\varphi\big((\lambda^{-1}x)\vee1\big)
      	+\widetilde q\,\varphi(\lambda x)\Big].
	\end{equation*}
If $\varphi$ is $L$-Lipschitz, then $\B_n\varphi$ is
	$\alpha\pi_{n-1}L$-Lipschitz.  
In particular, since $r>0$,
	\begin{equation*}
   		\operatorname{Lip}(\B_n\varphi)
   		\le \alpha\pi_{n-1}\operatorname{Lip}(\varphi)
   		<\operatorname{Lip}(\varphi)
	\end{equation*}
	whenever $\operatorname{Lip}(\varphi)>0$, 
	where $\operatorname{Lip}(f)$ denotes the Lipschitz constant of the function $f$.

The dynamic programming equation is
	$U^{[m]}_0(x)=g(x)\ (m\ge1),\ U^{[0]}_n(x)\equiv0,$
	with, for $n\ge1$,
	\begin{equation}\label{eq:dprussianrandom2}
  		U^{[m]}_n(x)=\max\Big\{\,g(x)+(\B_n U^{[m-1]}_{n-1})(x),\
                  (\B_n U^{[m]}_{n-1})(x)\,\Big\},
  		\qquad g(x)=x.
	\end{equation}
Define
	\begin{equation*}
  		\Delta U^{[m]}_n(x):=U^{[m]}_n(x)-U^{[m-1]}_n(x),
 		 \qquad
 	 	h^{[m]}_n(x):=(\B_n \Delta U^{[m]}_{n-1})(x)\quad(n\ge1),
	\end{equation*}
	$h^{[m]}_0(x):=0,\ h^{[0]}_n(x):=+\infty.$
Then
	\begin{equation*}
  		U^{[m]}_n(x)=\max\{g(x),h^{[m]}_n(x)\}+(\B_n U^{[m-1]}_{n-1})(x),
	\end{equation*}
	and exercise is an optimal action exactly when $x\ge h^{[m]}_n(x)$.
%%%%%%%%%%%%%%%%%%%%%%%%%%%%%%%%%%%%%%%%%%%%%%%
\begin{proposition}%[Random-maturity Russian structure]
\label{prop:russianrandomstructure}
Assume $r>0$.  
For arbitrary one-step survival probabilities
	$\pi_j\in[0,1]$ and all $n\ge0$, $m\ge1$:
\begin{enumerate}[(i)]
	\item for $m\ge2$,
      		\begin{align*}
        	\Delta U^{[m]}_n(x)=\med\{h^{[m]}_n(x),\,g(x),\,h^{[m-1]}_n(x)\},
		\end{align*}
     		 while $\Delta U^{[1]}_n(x)=\max\{g(x),h^{[1]}_n(x)\}$;
	\item $\Delta U^{[m]}_n(\cdot)\in\Mcl^+$ and $h^{[m]}_n(\cdot)\in\Mcl^+$; moreover,
      		$h^{[m]}_n(x)$ is $\alpha\pi_{n-1}$-Lipschitz for $n\ge1$;
	\item the marginal values are diminishing in the number of rights:
      		\begin{align*}
       			 \Delta U^{[m+1]}_n(x)\le\Delta U^{[m]}_n(x),
        			\qquad
       			 h^{[m+1]}_n(x)\le h^{[m]}_n(x);
		\end{align*}
	\item $U^{[m]}_n(x)$ is nonnegative, convex, nondecreasing and
      		$L_m$-Lipschitz with $L_m=1+\alpha+\cdots+\alpha^{m-1}$.
\end{enumerate}
If, in addition, Assumption~\ref{ass:A2} holds, then
	\begin{equation}\label{eq:russianrandomtime}
  		 \Delta U^{[m]}_n(x)\le\Delta U^{[m]}_{n+1}(x),
   		\qquad
   		h^{[m]}_n(x)\le h^{[m]}_{n+1}(x).
	\end{equation}
\end{proposition}
%%%%%%%%%%%%%%%%%%%%%%%%%%%%%%%%%%%%%%%%%%%%%%%
\begin{proof}
The proof is the time-inhomogeneous analogue of
	Proposition~\ref{prop:russianstructure}.  
At $n=0$, $\Delta U^{[1]}_0(x)=g(\cdot)\in\Mcl^+$ and $\Delta U^{[m]}_0(x)=0$ for $m\ge2$.
Suppose the assertions hold at $n-1$.  Since $\B_n$ is order preserving and
	maps $\Mcl^+$ into itself, with Lipschitz gain
	$\alpha\pi_{n-1}\le\alpha<1$, we have
	$h^{[m]}_n(\cdot)\in\Mcl^+$ and
	$h^{[m+1]}_n(x)\le h^{[m]}_n(x)$.  
The maximum formula for $m=1$ and the median
	identity for $m\ge2$ then show that every $\Delta U^{[m]}_n(x)$ belongs to
	$\Mcl^+$; monotonicity of the interval projection in each argument gives
	$\Delta U^{[m+1]}_n(x)\le\Delta U^{[m]}_n(x)$.  
This proves (i)--(iii).

Convexity and monotonicity in (iv) follow by induction from
	\eqref{eq:dprussianrandom2}, because $\B_n$ preserves these properties on
	nondecreasing convex functions.  
The Lipschitz estimate follows from
	\begin{align*}
  		\operatorname{Lip}_x\big(U^{[m]}_n(x)\big)
  		\le\max\Big\{1+\alpha\pi_{n-1}\operatorname{Lip}_x\big(U^{[m-1]}_{n-1}(x)\big),\
      		\alpha\pi_{n-1}\operatorname{Lip}_x\big(U^{[m]}_{n-1}(x)\big) \Big\},
	\end{align*}
	and the bound $\pi_{n-1}\le1$, using the same double induction as in
	Proposition~\ref{prop:russianstructure}(v).

Finally assume (A2). 
If $\Delta U^{[m]}_{n-1}(x)\le\Delta U^{[m]}_n(x)$, then positivity of $\B$ and
	$\pi_n\ge\pi_{n-1}$ give
	\begin{align*}
 	 	h^{[m]}_{n+1}(x)=\pi_n(\B \Delta U^{[m]}_n)(x)
  			\ge\pi_{n-1}(\B \Delta U^{[m]}_{n-1})(x)=h^{[m]}_n(x).
	\end{align*}
The maximum/median representation then yields
	$\Delta U^{[m]}_n(x)\le\Delta U^{[m]}_{n+1}(x)$.  
The base step is immediate, so
	\eqref{eq:russianrandomtime} follows by induction.
\end{proof}
%%%%%%%%%%%%%%%%%%%%%%%%%%%%%%%%%%%%%%%%%%%%%%%
\begin{theorem}%[Russian option with random maturity]
\label{thm:russianrandom}
Assume $r>0$ and define
	\begin{equation*}
   		y^{[m]*}_n:=\inf\big\{x\ge1:\ x\ge h^{[m]}_n(x)\big\}.
	\end{equation*}
Then:
\begin{enumerate}[(i)]
\item $y^{[m]*}_n$ is well defined and finite, and the exercise region is the upper interval
      \begin{align*}
         	\bar D^{[m]}_{R,n} :=\{x\ge1:\ x\ge h^{[m]}_n(x)\}=[y^{[m]*}_n,\infty);
	\end{align*}
\item the boundaries are nested in the number of remaining rights:
      \begin{align*}
         	y^{[m+1]*}_n\le y^{[m]*}_n;
	\end{align*}
\item if (A2) holds, then the boundary is nondecreasing in the remaining time:
      	\begin{align*}
         	1=y^{[m]*}_0\le y^{[m]*}_1\le\cdots\le y^{[m]*}_N;
	\end{align*}
\item starting from calendar time $0$, schedule exercises recursively by the
      	first entrance into the corresponding upper exercise regions.  
	More precisely, with $\min\emptyset:=\partial$,
      	\begin{align*}
       	 	\sigma_m^*=\min\{\nu\in\{0,\dots,N\}:\
              X_\nu\ge y^{[m]*}_{N-\nu}\},
	\end{align*}
      and for $i=m-1,\dots,1$,
      \begin{align*}
        	\sigma_i^*=\min\{\nu\in\{0,\dots,N\}:\
              \nu>\sigma_{i+1}^*,\
              X_\nu\ge y^{[i]*}_{N-\nu}\}.
	\end{align*}
The schedule is optimal for \eqref{eq:russianrandomobj}; an exercise
      pays only on $\{\sigma_i^*\le\widetilde N\}$, and if random maturity
      occurs before the next scheduled exercise, all remaining rights expire.
\end{enumerate}
\end{theorem}
%%%%%%%%%%%%%%%%%%%%%%%%%%%%%%%%%%%%%%%%%%%%%%%
\begin{proof}
For $n=0$, $h^{[m]}_0(x)=0$, so $y^{[m]*}_0=1$.  For $n\ge1$,
Proposition~\ref{prop:russianrandomstructure}(ii) gives
	\begin{align*}
  		 |h^{[m]}_n(x)-h^{[m]}_n(y)|
  		 \le\alpha\pi_{n-1}|x-y|.
	\end{align*}
Hence, for $1\le y<x$,
	\begin{align*}
 		\big(x-h^{[m]}_n(x)\big)-\big(y-h^{[m]}_n(y)\big)
			 \ge(1-\alpha\pi_{n-1})(x-y)>0.
	\end{align*}
Thus $x\mapsto x-h^{[m]}_n(x)$ is strictly increasing.  
Since $h^{[m]}_n(x)$ has at most linear growth with slope
	$\alpha\pi_{n-1}<1$, this difference tends to $+\infty$ as $x\to\infty$;
	therefore the exercise set is a nonempty upper interval, proving (i).

Part (ii) follows from $h^{[m+1]}_n(x)\le h^{[m]}_n(x)$.  
Under (A2),Proposition~\ref{prop:russianrandomstructure} gives
	$h^{[m]}_{n+1}(x)\ge h^{[m]}_n(x)$, and therefore
	$\{x:x\ge h^{[m]}_{n+1}(x)\}\subseteq
	 \{x:x\ge h^{[m]}_n(x)\}$, proving (iii).

Finally, at every live state the boundary action in (iv) attains the maximum
	in \eqref{eq:dprussianrandom2}.  
Successive applications of $\B_n$ generate
	exactly the cumulative survival factors $\Pi_{n,k}$ in
	\eqref{eq:russianrandomvalue}.  
Finite-horizon backward induction therefore
	proves optimality of the scheduled boundary rule, and
	\eqref{eq:russianrandomnumeraire} returns the corresponding value in the
	original numeraire.
\end{proof}
%%%%%%%%%%%%%%%%%%%%%%%%%%%%%%%%%%%%%%%%%%%%%%%
\begin{corollary}%[Fixed versus random maturity for the Russian option]
\label{cor:russianrandomcomparison}
Let $V^{[m]}_n(x)$, $f^{[m]}_n(x)$ and $x^{[m]*}_n$ denote the fixed-maturity Russian
	quantities of Proposition~\ref{prop:russianstructure} and
Theorem~\ref{thm:russian}.  Then, for every $m,n$,
	\begin{equation*}
   		U^{[m]}_n(x)\le V^{[m]}_n(x),
   		\qquad
  	 	h^{[m]}_n(x)\le f^{[m]}_n(x),
  	 	\qquad
   		y^{[m]*}_n\le x^{[m]*}_n.
	\end{equation*}
Thus random maturity lowers the Russian-option value and enlarges its exercise
	region; because the stopping region is an upper half-line, enlargement appears
	as a \emph{lower} exercise boundary.
\end{corollary}
%%%%%%%%%%%%%%%%%%%%%%%%%%%%%%%%%%%%%%%%%%%%%%%
\begin{proof}
We prove simultaneously by induction on $n$ that
$\Delta U^{[m]}_n(x)\le\Delta V^{[m]}_n(x)$ for every $m$.  At $n=0$ equality holds.
If the claim holds at $n-1$, then
	\begin{align*}
   		h^{[m]}_n(x)=\pi_{n-1}(\B \Delta U^{[m]}_{n-1})(x)
   		\le(\B \Delta V^{[m]}_{n-1})(x)=f^{[m]}_n(x).
	\end{align*}
For $m=1$ the maximum representation gives
$\Delta U^{[1]}_n(x)\le\Delta V^{[1]}_n(x)$; for $m\ge2$ the same conclusion follows
	from monotonicity of the median in each argument.  
Summing the marginal inequalities gives $U^{[m]}_n(x)\le V^{[m]}_n(x)$.  
Finally,
	$h^{[m]}_n(x)\le f^{[m]}_n(x)$ implies
	$[x^{[m]*}_n,\infty)\subseteq[y^{[m]*}_n,\infty)$, hence
	$y^{[m]*}_n\le x^{[m]*}_n$.
\end{proof}
%%%%%%%%%%%%%%%%%%%%%%%%%%%%%%%%%%%%%%%%%%%%%%%
\begin{corollary}\label{cor:russianrandomexamples}
For each of the maturity distributions in
	Examples~\ref{ex:uniform}--\ref{ex:poisson}, Assumption~\ref{ass:A2} holds.
Hence the random-maturity Russian boundaries satisfy both the nesting in the
	number of rights and the monotonicity in the remaining time asserted in
	Theorem~\ref{thm:russianrandom}.
\end{corollary}
%%%%%%%%%%%%%%%%%%%%%%%%%%%%%%%%%%%%%%%%%%%%%%%
%\begin{remark}%[Geometric killing]
%\label{rem:russiangeometric}
%Suppose, more specifically, that the contract survives each step independently
%	with a constant conditional probability $\pi\in(0,1]$ until the deterministic
%	cap $N$; equivalently,
%	\begin{align*}
%		\PP(\widetilde N=k)=(1-\pi)\pi^k\quad(0\le k<N),
%  		\qquad \PP(\widetilde N=N)=\pi^N.
%	\end{align*}
%Then $\pi_{n-1}\equiv\pi$ and
%	\begin{align*}
% 		\B_n=\pi\B, \qquad\operatorname{Lip}(\B_n)=\alpha\pi.
%	\end{align*}
%Thus independent geometric killing acts in the reduced Russian problem exactly
%	like an additional one-step discount factor.  
%The fixed-maturity model is recovered at $\pi=1$.
%\end{remark}

%%%%%%%%%%%%%%%%%%%%%%%%%%%%%%%%%%%%%%%%%%%%%%%
\section{Geometric-average Asian put}%: ratio reduction and multiple exercise}
\label{sec:asian}
%%%%%%%%%%%%%%%%%%%%%%%%%%%%%%%%%%%%%%%%%%%%%%%
The preceding two option families become one-dimensional after a suitable
	change of numeraire.  
The same idea also applies to a floating-strike Asian put
	when the running average is geometric.  
The resulting state process is no longer time-homogeneous, but its transition is explicit.  
This section gives the exact finite-lattice recursion, the marginal-value representation, 
	and the corresponding independent-random-maturity extension.

%%%%%%%%%%%%%%%%%%%%%%%%%%%%%%%%%%%%%%%%%%%%%%%
\subsection{Running geometric average and share-numeraire reduction}
%%%%%%%%%%%%%%%%%%%%%%%%%%%%%%%%%%%%%%%%%%%%%%%
For $\nu=0,\dots,N$ define the discrete running geometric average
	\begin{equation*}
   		G_\nu:=\left(\prod_{j=0}^{\nu}S_j\right)^{1/(\nu+1)}
   		=\exp\left\{\frac{1}{\nu+1}\sum_{j=0}^{\nu}\log S_j\right\}.
	\end{equation*}
A floating-strike geometric-average Asian put pays
	$(G_\nu-S_\nu)^+$ when a right is exercised at date $\nu$.  
Put
	\begin{equation*}
   		R_\nu:=\frac{G_\nu}{S_\nu}>0,\qquad g_A(x):=(x-1)^+.
	\end{equation*}
Then $(G_\nu-S_\nu)^+=S_\nu g_A(R_\nu)$.  
Hence the share measure $\Pt$ of Lemma~\ref{lem:numeraire} gives, 
	for every bounded stopping time $\sigma$,
	\begin{equation*}
  		\E\big[\alpha^\sigma(G_\sigma-S_\sigma)^+\big]
  		=s_0\,\Et\big[g_A(R_\sigma)\big].
	\end{equation*}
Thus the discount factor disappears after the numeraire change, exactly as for
	the Russian option.

The state $R$ is time-inhomogeneous.  
Set
	\begin{equation*}
   		a_\nu:=\frac{\nu+1}{\nu+2},\qquad \ell:=\log\lambda .
	\end{equation*}
%%%%%%%%%%%%%%%%%%%%%%%%%%%%%%%%%%%%%%%%%%%%%%%
\begin{lemma}%[Exact Asian transition]
\label{lem:asiantransition}
Under $\Pt$,
	\begin{equation}\label{eq:asiantransition}
  	R_{\nu+1}=R_\nu^{a_\nu}\lambda^{-a_\nu\varepsilon_{\nu+1}}
  		=\begin{cases}
     			\lambda^{-a_\nu}R_\nu^{a_\nu},&\varepsilon_{\nu+1}=+1,\\
     			\lambda^{ a_\nu}R_\nu^{a_\nu},&\varepsilon_{\nu+1}=-1,
   		\end{cases}
	\end{equation}
	with probabilities $\widetilde p$ and $\widetilde q$ from
	Lemma~\ref{lem:numeraire}.  More generally, for $0\le k<j\le N$,
	\begin{equation}\label{eq:asianmultistep}
  		\log R_j =\frac{k+1}{j+1}\log R_k-\frac{\ell}{j+1}\sum_{r=k+1}^{j}r\varepsilon_r .
	\end{equation}
Consequently $R$ is a one-dimensional time-inhomogeneous Markov chain.
\end{lemma}
%%%%%%%%%%%%%%%%%%%%%%%%%%%%%%%%%%%%%%%%%%%%%%%
\begin{proof}
From $G_{\nu+1}^{\nu+2}=G_\nu^{\nu+1}S_{\nu+1}$ and
$S_{\nu+1}=S_\nu\lambda^{\varepsilon_{\nu+1}}$,
\begin{align*}
 	\frac{G_{\nu+1}}{S_{\nu+1}}=\left(\frac{G_\nu}{S_\nu}\right)^{(\nu+1)/(\nu+2)}
   	\lambda^{-(\nu+1)\varepsilon_{\nu+1}/(\nu+2)},
\end{align*}
	which is \eqref{eq:asiantransition}.  
Multiplying the logarithmic recursion by $\nu+2$ and iterating gives \eqref{eq:asianmultistep}.
\end{proof}

For a nonnegative measurable function $\varphi$ define the one-step Asian operator
\begin{equation}\label{eq:asianoperator}
 	(\mathcal C_\nu\varphi)(x):=\widetilde p\,\varphi\!\left(\lambda^{-a_\nu}x^{a_\nu}\right)
   		+\widetilde q\,\varphi\!\left(\lambda^{ a_\nu}x^{a_\nu}\right),
 	\qquad 0\le\nu<N.
\end{equation}
It is positive and order preserving, and it preserves monotonicity because both
	state maps in \eqref{eq:asianoperator} are increasing.  
Formula \eqref{eq:asianmultistep} also gives an exact finite-sum representation for any
	multi-step transition.  
This is the discrete-time counterpart of the explicit Gaussian transition available 
	for the continuous-time geometric-average model.

%%%%%%%%%%%%%%%%%%%%%%%%%%%%%%%%%%%%%%%%%%%%%%%
\subsection{Multiple stopping, marginal values, and nested exercise sets}
%%%%%%%%%%%%%%%%%%%%%%%%%%%%%%%%%%%%%%%%%%%%%%%
Let $J_\nu^{[m]}(x)$ be the normalized value at calendar time $\nu$ when
	$R_\nu=x$ and $m$ rights remain.  
The original monetary value at time zero is $s_0J_0^{[m]}(1)$ because $G_0=S_0$.  
Since at most one right may be used at a date,
	\begin{equation*}
	 	J_N^{[m]}(x)=g_A(x)\quad(m\ge1),\qquad J_\nu^{[0]}(x)\equiv0,
	\end{equation*}
	and, for $0\le\nu<N$,
	\begin{equation}\label{eq:asiandp2}
		 J_\nu^{[m]}(x)=\max\Big\{g_A(x)+(\mathcal C_\nu J_{\nu+1}^{[m-1]})(x),\,
             (\mathcal C_\nu J_{\nu+1}^{[m]})(x)\Big\}.
	\end{equation}
Define the marginal value and the continuation premium by
	\begin{equation*}
		 \Delta J_\nu^{[m]}(x):=J_\nu^{[m]}(x)-J_\nu^{[m-1]}(x),\qquad
 			c_\nu^{[m]}(x):=(\mathcal C_\nu \Delta J_{\nu+1}^{[m]})(x)
 			\quad(\nu<N),
	\end{equation*}
	with $c_N^{[m]}(x):=0$ and $c_\nu^{[0]}(x):=+\infty$.  
Then
	\begin{equation}\label{eq:asiandp3}
 		J_\nu^{[m]}(x)=\max\{g_A(x),c_\nu^{[m]}(x)\}+(\mathcal C_\nu J_{\nu+1}^{[m-1]})(x).
	\end{equation}
%%%%%%%%%%%%%%%%%%%%%%%%%%%%%%%%%%%%%%%%%%%%%%%
\begin{proposition}%[Asian median identity and diminishing marginal values]
\label{prop:asianstructure}
For every $0\le\nu\le N$ and $m\ge1$:
\begin{enumerate}[(i)]
	\item $J_\nu^{[m]}(x)$, $\Delta J_\nu^{[m]}(x)$ and $c_\nu^{[m]}(x)$ are nonnegative
      		(where $c_N^{[m]}(x)=0$), and they are nondecreasing functions of the state;
	\item for $m\ge2$,
     	 \begin{equation}\label{eq:asianmedian}
        	\Delta J_\nu^{[m]}(x)=\med\{c_\nu^{[m]}(x),\,g_A(x),\,c_\nu^{[m-1]}(x)\},
      \end{equation}
      while $\Delta J_\nu^{[1]}(x)=\max\{g_A(x),c_\nu^{[1]}(x)\}$;
\item marginal values diminish with the number of rights:
      \begin{equation*}
       	 	\Delta J_\nu^{[m+1]}(x)\le\Delta J_\nu^{[m]}(x),
        	\qquad
        	c_\nu^{[m+1]}(x)\le c_\nu^{[m]}(x);
      \end{equation*}
\item the exercise sets
      \begin{equation*}
        	D_{A,\nu}^{[m]}:=\{x>0:g_A(x)\ge c_\nu^{[m]}(x)\}
      \end{equation*}
      are nested:
      	$D_{A,\nu}^{[m]}\subseteq D_{A,\nu}^{[m+1]}$;
\item all values are finite and have at most linear growth.  Moreover, for
      $\nu<N$,
      \begin{equation}\label{eq:asiansublinear}
          	c_\nu^{[m]}(x)=O(x^{a_\nu})=o(x)\qquad(x\to\infty).
      \end{equation}
\end{enumerate}
\end{proposition}
%%%%%%%%%%%%%%%%%%%%%%%%%%%%%%%%%%%%%%%%%%%%%%%
\begin{proof}
At the terminal date,
	$\Delta J_N^{[1]}(x)=g_A(x)$ and $\Delta J_N^{[m]}(x)=0$ for $m\ge2$, so all assertions
	start in the required order.  
Suppose they hold at $\nu+1$.  
Positivity and order preservation of $\mathcal C_\nu$ give
	$c_\nu^{[m+1]}(x)\le c_\nu^{[m]}(x)$ and preserve monotonicity in the state.
Subtracting \eqref{eq:asiandp3} at levels $m$ and $m-1$ yields the same interval-projection
	algebra as Lemma~\ref{lem:median}, hence \eqref{eq:asianmedian}.  
Monotonicity of the median in each argument gives
	$\Delta J_\nu^{[m+1]}(x)\le\Delta J_\nu^{[m]}(x)$.  
This proves (i)--(iii) by backward induction.  
Part (iv) follows immediately from $c_\nu^{[m+1]}(x)\le c_\nu^{[m]}(x)$.

For (v), $g_A(x)\le x$.  
If a function has at most linear growth, then
	\eqref{eq:asianoperator} and $a_\nu<1$ show that its image under
	$\mathcal C_\nu$ is bounded by a constant multiple of $1+x^{a_\nu}$.  
Backward induction in \eqref{eq:asiandp2} therefore gives at most linear growth of every
	$J_\nu^{[m]}(x)$ and $\Delta J_\nu^{[m]}(x)$, and then
	\eqref{eq:asiansublinear} follows directly from \eqref{eq:asianoperator}.
\end{proof}

Unlike the put and Russian operators, $\mathcal C_\nu$ acts through the
	fractional power $x^{a_\nu}$, so the global $1$-Lipschitz estimate used in the
	preceding models does not follow from the same argument.  
The fractional-power structure itself, however, yields a multiplicative scaling inequality.  
This inequality is sufficient to prove the single-crossing property needed for a
	threshold theorem.
%%%%%%%%%%%%%%%%%%%%%%%%%%%%%%%%%%%%%%%%%%%%%%%
\begin{lemma}%[Asian scaling inequality]
\label{lem:asianscaling}
For every $q\ge1$, $x>0$, $m\ge1$, and $0\le\nu\le N$,
	\begin{equation}\label{eq:asianscalingmarginal}
   		\Delta J_\nu^{[m]}(qx)+1\le q\bigl(\Delta J_\nu^{[m]}(x)+1\bigr).
	\end{equation}
Moreover, if $\nu<N$, then
	\begin{equation}\label{eq:asianscalingcont}
   		c_\nu^{[m]}(qx)+1 \le q^{a_\nu}\bigl(c_\nu^{[m]}(x)+1\bigr).
	\end{equation}
\end{lemma}
%%%%%%%%%%%%%%%%%%%%%%%%%%%%%%%%%%%%%%%%%%%%%%%
\begin{proof}
We argue backward in $\nu$.  
At the terminal date,
	\begin{align*}
   		\Delta J_N^{[1]}(x)=g_A(x)=(x-1)^+,\qquad
   		\Delta J_N^{[m]}(x)=0\quad(m\ge2).
	\end{align*}
Hence
	\begin{align*}
 		g_A(qx)+1=\max\{qx,1\} \le q\max\{x,1\}=q\bigl(g_A(x)+1\bigr), 
	\end{align*}
	and \eqref{eq:asianscalingmarginal} also holds for $m\ge2$ because $1\le q$.

Suppose now that \eqref{eq:asianscalingmarginal} holds at date $\nu+1$ for
	every $m$.  Write the two state maps in \eqref{eq:asiantransition} as
	\begin{align*}
   		T_{\nu,\pm}(x):=\lambda^{\pm a_\nu}x^{a_\nu}.
	\end{align*}
Then
	\begin{align*}
   		T_{\nu,\pm}(qx)=q^{a_\nu}T_{\nu,\pm}(x).
	\end{align*}
Using the induction hypothesis pointwise in \eqref{eq:asianoperator} gives
	\begin{align*}
 		c_\nu^{[m]}(qx)+1 &=(\mathcal C_\nu\Delta J_{\nu+1}^{[m]})(qx)+1\\
 			&\le q^{a_\nu}\bigl((\mathcal C_\nu\Delta J_{\nu+1}^{[m]})(x)+1\bigr)
 			=q^{a_\nu}\bigl(c_\nu^{[m]}(x)+1\bigr),
	\end{align*}
	which proves \eqref{eq:asianscalingcont}.

Since $a_\nu\in(0,1)$, we have $q^{a_\nu}\le q$, and also
	\begin{align*}
   		g_A(qx)+1\le q\bigl(g_A(x)+1\bigr).
	\end{align*}
For $m=1$, the identity
	$\Delta J_\nu^{[1]}(x)=\max\{g_A(x),c_\nu^{[1]}(x)\}$ and monotonicity of the maximum
	give \eqref{eq:asianscalingmarginal}.  
For $m\ge2$, use the median identity in
	Proposition~\ref{prop:asianstructure}(ii) together with
	\begin{align*}
   		\med\{u,v,w\}+1=\med\{u+1,v+1,w+1\}.
	\end{align*}
The median is nondecreasing in each argument and commutes with multiplication
	by a positive constant.  
Applying the preceding bounds to $c_\nu^{[m]}(x)$, $g_A(x)$ 
	and $c_\nu^{[m-1]}$ therefore yields \eqref{eq:asianscalingmarginal}.  
This closes the backward induction.
\end{proof}
%%%%%%%%%%%%%%%%%%%%%%%%%%%%%%%%%%%%%%%%%%%%%%%
\begin{corollary}%[Asian single crossing]
\label{cor:asiansinglecross}
Let $\nu<N$ and $m\ge1$.  If for some $x\ge1$,
	$g_A(x)\ge c_\nu^{[m]}(x),$
	then, for every $y>x$,
	$g_A(y)>c_\nu^{[m]}(y).$
Consequently $D_{A,\nu}^{[m]}\cap[1,\infty)$ is upward closed.
\end{corollary}
%%%%%%%%%%%%%%%%%%%%%%%%%%%%%%%%%%%%%%%%%%%%%%%
\begin{proof}
Set $q:=y/x>1$.  Since $g_A(x)=x-1$ for $x\ge1$, Lemma~\ref{lem:asianscaling} gives
	\begin{align*}
  		c_\nu^{[m]}(y)+1
  		\le q^{a_\nu}\bigl(c_\nu^{[m]}(x)+1\bigr)
  		\le q^{a_\nu}x< qx =y=g_A(y)+1 .
	\end{align*}
The strict inequality uses $a_\nu<1$ and $q>1$.  
Hence $c_\nu^{[m]}(y)<g_A(y)$.
\end{proof}
%%%%%%%%%%%%%%%%%%%%%%%%%%%%%%%%%%%%%%%%%%%%%%%
\begin{theorem}%[Geometric-Asian put: optimal threshold policy]
\label{thm:asianthreshold}
Define
	\begin{equation*}
 		b_{A,\nu}^{[m]}:=\inf\{x\ge1:g_A(x)\ge c_\nu^{[m]}(x)\}.
	\end{equation*}
Then $b_{A,\nu}^{[m]}<\infty$ and the economically relevant exercise region is
	\begin{align*}
   		D_{A,\nu}^{[m]}\cap[1,\infty)=[b_{A,\nu}^{[m]},\infty).
	\end{align*}
Furthermore
	\begin{equation*}
   		b_{A,\nu}^{[m+1]}\le b_{A,\nu}^{[m]},
	\end{equation*}
	and the recursive first-entry rule into these upper regions is optimal.
At the terminal date $b_{A,N}^{[m]}=1$.
\end{theorem}
%%%%%%%%%%%%%%%%%%%%%%%%%%%%%%%%%%%%%%%%%%%%%%%
\begin{proof}
At $\nu=N$, $c_N^{[m]}(x)\equiv0$, so the assertion is immediate.  
Let $\nu<N$.  
Proposition~\ref{prop:asianstructure}(v) gives
	$c_\nu^{[m]}(x)=o(x)$, whereas $g_A(x)=x-1$ for $x\ge1$.  
Hence
	\begin{align*}
   		g_A(x)-c_\nu^{[m]}(x)\longrightarrow+\infty   \qquad(x\to\infty),
	\end{align*}
	so $D_{A,\nu}^{[m]}\cap[1,\infty)$ is nonempty.  
Continuity of $c_\nu^{[m]}$ follows by backward induction from \eqref{eq:asiandp2}, so this
	set is closed.  
Corollary~\ref{cor:asiansinglecross} shows that it is also upward closed.  
Therefore
	\begin{align*}
   		D_{A,\nu}^{[m]}\cap[1,\infty)=[b_{A,\nu}^{[m]},\infty)
	\end{align*}
	for a finite $b_{A,\nu}^{[m]}$.  
Boundary nesting follows from $c_\nu^{[m+1]}(x)\le c_\nu^{[m]}(x)$, equivalently
	$D_{A,\nu}^{[m]}\subseteq D_{A,\nu}^{[m+1]}$.
Finally, the boundary action attains the maximum in \eqref{eq:asiandp2} at every
	state, so finite-horizon backward induction, or equivalently
	Theorem~\ref{thm:multiple} applied to the time-space Markov chain
	$(\nu,R_\nu)$, proves optimality.
\end{proof}

\begin{remark}
No general monotonicity of $\nu\mapsto b_{A,\nu}^{[m]}$ is asserted here.  
The transition operator itself changes with calendar time through
	$a_\nu=(\nu+1)/(\nu+2)$, so the time-monotonicity argument used for the
	homogeneous put and Russian chains does not transfer without additional
	conditions.
\end{remark}

%%%%%%%%%%%%%%%%%%%%%%%%%%%%%%%%%%%%%%%%%%%%%%%
\subsection{Random maturity}
%%%%%%%%%%%%%%%%%%%%%%%%%%%%%%%%%%%%%%%%%%%%%%%
Let the independent random maturity $\widetilde N$ be the same as in
	Section~\ref{sec:random}.  
It remains independent of the stock process under $\Pt$, 
	because the Radon--Nikodym density defining $\Pt$ is measurable with
	respect to the stock filtration.  In calendar time write
	\begin{equation*}
 		\rho_\nu:=\PP(\widetilde N\ge\nu+1\mid\widetilde N\ge\nu)
		 	=\pi_{N-\nu-1},\qquad 0\le\nu<N,
	\end{equation*}
	and
	\begin{equation*}
	 	\Pi_{k,j}:=\prod_{r=k}^{j-1}\rho_r=\PP(\widetilde N\ge j\mid\widetilde N\ge k),
 		\qquad k\le j\le N,
	\end{equation*}
	with the empty product equal to one.

Conditionally on the contract being alive at calendar time $\nu$, define the
	survival-weighted Asian operator
	\begin{equation*}
   		\bar{\mathcal C}_\nu:=\rho_\nu\mathcal C_\nu .
	\end{equation*}
The original monetary value at time zero becomes
	\begin{equation*}
		\sup_{\vec\sigma\in\T_0^{[m]}}\E\Big[\sum_{i:\,\sigma_i\le N}
     		 \alpha^{\sigma_i}\mathbf 1_{\{\sigma_i\le\widetilde N\}}
     		 (G_{\sigma_i}-S_{\sigma_i})^+\Big]=s_0\,\bar J_0^{[m]}(1).
	\end{equation*}
The normalized alive-state values satisfy
	$\bar J_N^{[m]}(x)=g_A(x),\ \bar J_\nu^{[0]}(x)\equiv0,$
	and, for $\nu<N$,
	\begin{equation}\label{eq:asianrandomdp2}
 		\bar J_\nu^{[m]}(x)=\max\Big\{g_A(x)+(\bar{\mathcal C}_\nu \bar J_{\nu+1}^{[m-1]})(x),\,
             (\bar{\mathcal C}_\nu \bar J_{\nu+1}^{[m]})(x)\Big\}.
	\end{equation}
Indeed, the immediate payoff is not multiplied by $\rho_\nu$ because the
	contract is already known to be alive at date $\nu$; only future values require
	survival to the next date.

Set
	\begin{equation*}
 		\Delta\bar J_\nu^{[m]}(x):=\bar J_\nu^{[m]}(x)-\bar J_\nu^{[m-1]}(x),\qquad
		 \bar c_\nu^{[m]}(x) :=(\bar{\mathcal C}_\nu \Delta\bar J_{\nu+1}^{[m]})(x)
	 	\quad(\nu<N),
	\end{equation*}
	with $\bar c_N^{[m]}(x)=0$.
%%%%%%%%%%%%%%%%%%%%%%%%%%%%%%%%%%%%%%%%%%%%%%%
\begin{proposition}%[Geometric-Asian put: random-maturity comparison]
\label{prop:asianrandom}
For arbitrary survival probabilities $\rho_\nu\in[0,1]$:
\begin{enumerate}[(i)]
	\item the median identity and diminishing-marginal-value conclusions of
      		Proposition~\ref{prop:asianstructure} hold with bars;
	\item the random-maturity exercise sets
      		\begin{align*}
         		\bar D_{A,\nu}^{[m]} :=\{x>0:g_A(x)\ge\bar c_\nu^{[m]}(x)\}
		\end{align*}
      		are nested in $m$;
	\item for every $\nu,m$,
      		$\Delta\bar J_\nu^{[m]}(x)\le\Delta J_\nu^{[m]}(x),\
       				 \bar c_\nu^{[m]}(x)\le c_\nu^{[m]}(x),\text{and} \ 
       				 \bar J_\nu^{[m]}(x)\le J_\nu^{[m]}(x);$
      consequently
      \begin{equation}\label{eq:asianrandomsetinclude}
        	D_{A,\nu}^{[m]}\subseteq\bar D_{A,\nu}^{[m]}.
      \end{equation}
\end{enumerate}
\end{proposition}
%%%%%%%%%%%%%%%%%%%%%%%%%%%%%%%%%%%%%%%%%%%%%%%
\begin{proof}
Part (i) is the same backward induction as in
	Proposition~\ref{prop:asianstructure}, because
	$\bar{\mathcal C}_\nu$ is positive and order preserving.  
For the comparison, argue simultaneously backward in $\nu$.  
At $\nu=N$ the marginal values agree.
If $\Delta\bar J_{\nu+1}^{[m]}(x)\le\Delta J_{\nu+1}^{[m]}(x)$, then
	\begin{align*}
	 	\bar c_\nu^{[m]}(x)=\rho_\nu(\mathcal C_\nu \Delta\bar J_{\nu+1}^{[m]})(x)
	 	\le(\mathcal C_\nu \Delta J_{\nu+1}^{[m]})(x)=c_\nu^{[m]}(x).
	\end{align*}
For $m=1$ the maximum formula preserves the inequality, and for $m\ge2$ the
	same is true by monotonicity of the median in each argument.  
Hence $\Delta\bar J_\nu^{[m]}(x)\le\Delta J_\nu^{[m]}(x)$.  
Summing over the marginal rights gives the value comparison.  
Finally $\bar c_\nu^{[m]}(x)\le c_\nu^{[m]}(x)$ gives
	\eqref{eq:asianrandomsetinclude} directly.
\end{proof}
%%%%%%%%%%%%%%%%%%%%%%%%%%%%%%%%%%%%%%%%%%%%%%%
\begin{lemma}%[Asian scaling under random maturity]
\label{lem:asianrandomscaling}
For every $q\ge1$, $x>0$, $m\ge1$, and $0\le\nu\le N$,
	\begin{equation}\label{eq:asianrandomscalingmarginal}
  		\Delta\bar J_\nu^{[m]}(qx)+1
  		\le q\bigl(\Delta\bar J_\nu^{[m]}(x)+1\bigr).
	\end{equation}
Moreover, if $\nu<N$, then
	\begin{equation*}
   		\bar c_\nu^{[m]}(qx)+1
   		\le q^{a_\nu}\bigl(\bar c_\nu^{[m]}(x)+1\bigr).
	\end{equation*}
\end{lemma}
%%%%%%%%%%%%%%%%%%%%%%%%%%%%%%%%%%%%%%%%%%%%%%%
\begin{proof}
At the terminal date the argument is the same as in the fixed-maturity case.
Assume \eqref{eq:asianrandomscalingmarginal} at date $\nu+1$.  
The same scaling of the state maps as in Lemma~\ref{lem:asianscaling} gives
	\begin{align*}
 		\mathcal C_\nu\Delta\bar J_{\nu+1}^{[m]}(qx)+1
  		\le q^{a_\nu}\bigl(\mathcal C_\nu\Delta\bar J_{\nu+1}^{[m]}(x)+1 \bigr).
	\end{align*}
Hence, using $0\le\rho_\nu\le1$ and $q^{a_\nu}\ge1$,
	\begin{align*}
 		\bar c_\nu^{[m]}(qx)+1
 			&=\rho_\nu\mathcal C_\nu\Delta\bar J_{\nu+1}^{[m]}(qx)+1\\
 			&\le\rho_\nu q^{a_\nu}
 			\bigl(\mathcal C_\nu\Delta\bar J_{\nu+1}^{[m]}(x)+1\bigr)+(1-\rho_\nu)\\
 			&\le q^{a_\nu}\bigl(\rho_\nu\mathcal C_\nu\Delta\bar J_{\nu+1}^{[m]}(x)+1 \bigr)
			=q^{a_\nu}\bigl(\bar c_\nu^{[m]}(x)+1\bigr).
	\end{align*}
Applying the same maximum/median argument as in the fixed-maturity proof,
	using Proposition~\ref{prop:asianrandom}(i), yields
	\eqref{eq:asianrandomscalingmarginal} and closes the backward induction.
\end{proof}
%%%%%%%%%%%%%%%%%%%%%%%%%%%%%%%%%%%%%%%%%%%%%%%
\begin{corollary}%[Geometric-Asian put with random maturity]
\label{cor:asianrandomthreshold}
There is a finite threshold $\bar b_{A,\nu}^{[m]}$ such that
	\begin{align*}
 	 	\bar D_{A,\nu}^{[m]}\cap[1,\infty)=[\bar b_{A,\nu}^{[m]},\infty),
		\qquad \bar b_{A,\nu}^{[m+1]}\le\bar b_{A,\nu}^{[m]}.
	\end{align*}
Moreover,
	\begin{equation}\label{eq:asianboundarycompare}
   		\bar b_{A,\nu}^{[m]}\le b_{A,\nu}^{[m]}.
	\end{equation}
Thus independent random maturity enlarges the geometric-Asian stopping region.
\end{corollary}
%%%%%%%%%%%%%%%%%%%%%%%%%%%%%%%%%%%%%%%%%%%%%%%
\begin{proof}
The case $\nu=N$ is immediate, so let $\nu<N$.  Suppose that for some $x\ge1$,
	$g_A(x)\ge\bar c_\nu^{[m]}(x)$, and let $y>x$ with $q:=y/x>1$.
Lemma~\ref{lem:asianrandomscaling} gives
\begin{align*}
	\begin{aligned}
  	\bar c_\nu^{[m]}(y)+1
  		\le q^{a_\nu}\bigl(\bar c_\nu^{[m]}(x)+1\bigr)
  		\le q^{a_\nu}x< qx =y=g_A(y)+1 .
	\end{aligned}
\end{align*}
Thus the random-maturity model has the same single-crossing property, and
	$\bar D_{A,\nu}^{[m]}\cap[1,\infty)$ is upward closed.

The growth argument in Proposition~\ref{prop:asianstructure}(v) is unchanged
	under $\bar{\mathcal C}_\nu=\rho_\nu\mathcal C_\nu$ with
	$0\le\rho_\nu\le1$, so $\bar c_\nu^{[m]}(x)=o(x)$.  
Continuity follows by backward induction from \eqref{eq:asianrandomdp2}.  
Hence the stopping set is a closed upper interval beginning at a finite threshold.  
Boundary nesting in the number of rights follows from Proposition~\ref{prop:asianrandom}(ii).
Finally, \eqref{eq:asianrandomsetinclude} and the upper-interval
	representations of the two stopping sets imply \eqref{eq:asianboundarycompare}.
\end{proof}

%%%%%%%%%%%%%%%%%%%%%%%%%%%%%%%%%%%%%%%%%%%%%%%
\section{Conclusion}\label{sec:conclusion}
%%%%%%%%%%%%%%%%%%%%%%%%%%%%%%%%%%%%%%%%%%%%%%%
The main conclusion of this paper is that the sequence of marginal values
	$\Delta V^{[m]}$, together with the median identity
	$\Delta V^{[m]}_n(x)=\med\{f^{[m]}_n(x),\,g(x),\,f^{[m-1]}_n(x)\},$
	provides the key structural tool for deriving the optimal multiple-exercise
	rule for the American put.
The same marginal-value approach, combined with appropriate state reductions,
	also applies to Russian and geometric-average Asian options.
	
%%%%%%%%%%%%%%%%%%%%%%%%%%%%%%%%%%%%%%%%%%%%%%%	
\appendix
\section{The median operation and the class $\Mcl$}\label{app:operator}
%%%%%%%%%%%%%%%%%%%%%%%%%%%%%%%%%%%%%%%%%%%%%%%
For real numbers $a\le c$ and arbitrary $t$, define the projection of $t$
	onto the interval $[a,c]$ by
	$\Pi_{[a,c]}(t):=\min\{c,\max\{t,a\}\}$.  
Then
	$\med\{a,t,c\}=\Pi_{[a,c]}(t)$ whenever $a\le c$.
%%%%%%%%%%%%%%%%%%%%%%%%%%%%%%%%%%%%%%%%%%%%%%%
\begin{lemma}\label{lem:medianproperties}
Let $a\le c$ and $a'\le c'$ be real numbers and $t,t'\in\R$.
	\begin{enumerate}[(i)]
	\item $\Pi_{[a,c]}(t)$ is nondecreasing in each of $a$, $c$ and $t$.  In
      		particular, if $a\le a'$, $c\le c'$ and $t\le t'$, then
      		$\Pi_{[a,c]}(t)\le\Pi_{[a',c']}(t')$.
	\item $|\Pi_{[a,c]}(t)-\Pi_{[a',c']}(t')|
     		 \le\max\{|a-a'|,|c-c'|,|t-t'|\}$.
	\item If $a,c,t\ge0$ then $\Pi_{[a,c]}(t)\ge0$.
	\end{enumerate}
\end{lemma}
%%%%%%%%%%%%%%%%%%%%%%%%%%%%%%%%%%%%%%%%%%%%%%%
\begin{proof}
(i) is clear from the formula, both $\min$ and $\max$ being nondecreasing in
	each argument.  
(ii) follows from the fact that $\min$ and $\max$ of two
	$1$-Lipschitz maps are $1$-Lipschitz.  
(iii) is clear.
\end{proof}
%%%%%%%%%%%%%%%%%%%%%%%%%%%%%%%%%%%%%%%%%%%%%%%
\begin{corollary}\label{cor:Mmedian}
Let $\varphi_1(\cdot),\varphi_2(\cdot),\varphi_3(\cdot)\in\Mcl$ satisfy
	$\varphi_1(x)\le\varphi_3(x)$ for all $x$, and set
	$\mathsf m(x):=\med\{\varphi_1(x),\varphi_2(x),\varphi_3(x)\}$.  
Then $\mathsf m(\cdot)\in\Mcl$, and $\mathsf m(x)$ is nondecreasing in each of
	$\varphi_1(x)$, $\varphi_2(x)$ and $\varphi_3(x)$.  
The same holds with $\Mcl^{+}$ in place of $\Mcl$.
\end{corollary}
%%%%%%%%%%%%%%%%%%%%%%%%%%%%%%%%%%%%%%%%%%%%%%%
\begin{proof}
Nonnegativity is Lemma~\ref{lem:medianproperties}(iii).  
For $x>y$, apply Lemma~\ref{lem:medianproperties}(ii) to
	\begin{align*}
	 	(a,c,t)=(\varphi_1(x),\varphi_3(x),\varphi_2(x)),\qquad
		 (a',c',t')=(\varphi_1(y),\varphi_3(y),\varphi_2(y)).
	\end{align*}
It gives
	$|\mathsf m(x)-\mathsf m(y)|\le\max_i|\varphi_i(x)-\varphi_i(y)|\le|x-y|$; and
	Lemma~\ref{lem:medianproperties}(i) gives $\mathsf m(x)\le\mathsf m(y)$ for $x>y$ when all
	$x\mapsto\varphi_i(x)$ are nonincreasing (respectively $\ge$ when all are nondecreasing).
\end{proof}
%%%%%%%%%%%%%%%%%%%%%%%%%%%%%%%%%%%%%%%%%%%%%%%
\section{Computational details for Remark~\ref{rem:crr}}\label{app:smoothfit}
%%%%%%%%%%%%%%%%%%%%%%%%%%%%%%%%%%%%%%%%%%%%%%%
The value functions in Remark~\ref{rem:crr} were computed with
	$1+r=e^{\rho T/N}$ and
	$\lambda=e^{\sigma\sqrt{T/N}}$ by backward induction on the lattice
	$\{x\lambda^{j}:|j|\le N\}$ generated by the evaluation point $x$, using
	\eqref{eq:dp2} with $V^{[m]}_0(x)=g(x)$; this is exact arithmetic up to
	floating-point error, no interpolation being involved, because the lattice
	generated by $x$ is closed under $y\mapsto\lambda^{\pm1}y$.  
The boundary $b$ was located by bisection on $x\mapsto g(x)-f^{[m]}_n(x)$, which is monotone by
	Theorem~\ref{thm:main}(i), to a tolerance of $10^{-12}$, and the one-sided
	derivatives were evaluated by one-sided difference quotients with increment
	$10^{-6}$; since the value function is piecewise affine
	(Lemma~\ref{lem:pwaffine}) and the nearest breakpoint is at distance of order
	$10^{-2}$ in all the reported cases, the quotients reproduce the exact
	one-sided slopes to the digits shown.  
The same routine, run with $\lambda=1.2$, $r=0.05$, $K=1$, reproduces the values in
	Example~\ref{ex:nonconvex} and Proposition~\ref{ex:nosmoothfit}. 
These computations are used only as numerical checks and illustrations; none of the
	proofs above relies on them.

%%%%%%%%%%%%%%%%%%%%%%%%%%%%%%%%%%%%%%%%%%
\section*{Declaration of Generative AI and AI-Assisted Technologies in the Manuscript Preparation Process}
In preparing this work, the author used ChatGPT (OpenAI) to assist with English-language editing and to check 		selected numerical calculations. 
The author developed all mathematical arguments and verified the final content.
%%%%%%%%%%%%%%%%%%%%%%%%%%%%%%%%%%%%%%%%%%

\end{document}